\documentclass[11pt]{article}
\usepackage{graphicx} % Required for inserting images
\usepackage{xcolor}
\usepackage{amsthm}
\usepackage{amsmath}
\usepackage{mathtools}
\usepackage{slashbox}
\usepackage{amssymb}% http://ctan.org/pkg/amssymb
\usepackage{pifont}% http://ctan.org/pkg/pifont
\usepackage{verbatim}
\usepackage{faktor}
\usepackage{caption}
\usepackage{subcaption}
\usepackage{geometry}

\newcommand{\isEquivTo}[1]{\underset{#1}{\sim}}
\newcommand{\littleo}[1]{\underset{#1}{o}}

\makeatletter
\newtheorem*{rep@theorem}{\rep@title}
\newcommand{\newreptheorem}[2]{%
\newenvironment{rep#1}[1]{%
 \def\rep@title{#2 \ref{##1}}%
 \begin{rep@theorem}}%
 {\end{rep@theorem}}}
\makeatother

\newtheorem{thm}{Theorem}[section]
\newtheorem{lemma}[thm]{Lemma}
\newtheorem{cor}[thm]{Corollary}
\newtheorem{prop}[thm]{Propostion}
\newreptheorem{prop}{Propostion}
\newtheorem*{rmk}{Remark}
\newtheorem*{definition}{Definition}

\title{Rotation number for a homogeneous vector field in $\mathbf{C}^{2}$}
\author{Adrien Kachkachi}
\date{May 2026}

\begin{document}

\maketitle

\begin{abstract}
    In this paper, we use affine surfaces to describe completely the real-time trajectories of a homogeneous vector field in $\mathbf{C}^{2}$. We prove the existence of a continuous ``rotation number" in $\mathbf{C}^{2}$ which is constant along real-time trajectories. A description of the trajectories depending on their rotation number is given.
\end{abstract}

\section{Introduction}

\label{sec:Intro}

\subsection{Generalities on homogeneous vector fields}

In this paper, we will be concerned about the real-time dynamics of vector fields $\mathbf{v}=\left(v_{1},v_{2}\right)$ in $\mathbf{C}^{2}$, where $v_{1}$ and $v_{2}$ are homogeneous polynomials of degree $k\geq2$. We say that such vector fields have degree $k$. The solutions of the induced ordinary differential equation
\begin{equation}
\label{eq:eq_diff}
    \gamma'(t)=\mathbf{v}\circ\gamma(t)
\end{equation}
have been extensively studied as functions defined on a neighborhood of 0 in $\mathbf{C}$, but our underlying motivation is the study of the germs tangent to the identity of the form $$F\left(x,y\right)=\left(
\begin{array}{c}
    x \\
    y
\end{array}
\right)+
\left(
\begin{array}{c}
    v_{1}\left(x,y\right) \\
    v_{2}\left(x,y\right)
\end{array}
\right)+\text{higher order terms}.$$ Even though the study of such germs is out of the scope of this paper, the solutions of (\ref{eq:eq_diff}) defined in a neighborhood of 0 in $\mathbf{R}$ can provide us with some intuition about the dynamics of the germs, and in particular of the time-one flows associated to the vector fields.

\begin{definition}
    A maximal solution of (\ref{eq:eq_diff}) defined on a sub-interval of $\mathbf{R}$ is called a (real-time) trajectory.
\end{definition}

Let $\mathbf{v}$ be a homogeneous vector field of degree $k$ in $\mathbf{C}^{2}$. The polynomial $P\left(x,y\right):=x\,v_{2}\left(x,y\right)-y\,v_{1}\left(x,y\right)$ is either identically zero or homogeneous of degree $k+1$. In the former case, $\mathbf{v}$ is called \textit{dicritical}. We will assume $\mathbf{v}$ is non-dicritical, hence the set of solutions to the equation $P\left(x,y\right)=0$ consists of a union of $k+1$ complex lines $L_{0},\ldots,L_{k}$ counted with multiplicities. These lines are referred to as \textit{characteristic directions}. Their union $\mathcal{C}:=L_{0}\cup\ldots\cup L_{k}$ is called the characteristic cone. Let $\pi:\mathbf{C}^{2}\setminus\{(0,0)\}\rightarrow\mathbf{CP}^{1}$ defined by $\pi(x,y):=[x:y]$, the complex line passing through $(x,y)$ in $\mathbf{C}^{2}$. Let also $h:\mathbf{CP}^{1}\rightarrow\widehat{\mathbf{C}}$ be the isomorphism between $\mathbf{CP}^{1}$ and the Riemann sphere defined by $h([x:y]):=x/y$ if $y\neq0$, and $h([1:0]):=\infty$. The map $g:=h\circ\pi:\mathbf{C}^{2}\setminus\{(0,0)\}\rightarrow\widehat{\mathbf{C}}$ maps the characteristic directions to $k+1$ points $\left\{z_{0},\ldots,z_{k}\right\}$. The following theorem, due to Abate and Tovena, sheds light on the connection between homogeneous vector fields of $\mathbf{C}^{2}$ and (meromorphic) affine surfaces (see the definition below).

\begin{thm}[\cite{AbateTovena}]
\label{thm:AbateTovena}
    There exists a meromorphic affine structure on $\widehat{\mathbf{C}}$, denoted $\mathcal{S}_{\mathbf{v}}$, with singularities at $z_{0},\ldots,z_{k}$, such that $g_{|\mathbf{C}^{2}\setminus\mathcal{C}}$ sends real-time trajectories for $\mathbf{v}$ to geodesics of $\mathcal{S}_{\mathbf{v}}$.
\end{thm}

Conversely, to any meromorphic affine structure on $\widehat{\mathbf{C}}$ we can associate a homogeneous vector field on $\mathbf{C}^{2}$ that induces the original meromorphic affine structure. We shall explain how to go from the vector field to the associated affine surface, and vice versa.

\subsection{Generalities on affine surfaces}

In this section we recall basic notions about affine surfaces. The interested reader will find useful references in \cite{DuryevFougeronGhazouani}.

\begin{definition}
    An affine surface is a topological surface $S$ equipped with an atlas whose changes of charts are required to be complex affine maps $z\mapsto az+b$, $a\in\mathbf{C}^{*}$, $b\in\mathbf{C}$. The surface $S$ is said to be a dilation surface if $a\in\mathbf{R}_{+}$, and flat if $\lvert a\rvert=1$. A flat dilation surface is called ``translation surface".
\end{definition}

An affine surface can be fairly well understood through the concept of \textit{non-linearity} (see, for instance, the exposition made in \cite{BuffRaissy}). Roughly, the non-linearity of a holomorphic function between two affine surfaces measures how far the given function is from being affine. More precisely, let $\mathcal{S}_{1}$ and $\mathcal{S}_{2}$ be two affine structures, and $f:\mathcal{S}_{1}\rightarrow \mathcal{S}_{2}$ be a holomorphic function with non-vanishing differential. Let also $\varphi_{1}$ and $\varphi_{2}$ be affine charts around $s\in\mathcal{S}_{1}$ and $f\left(s\right)\in\mathcal{S}_{2}$, respectively. Then, the non-linearity of $f$, denoted $\mathcal{N}_{f}$, is the holomorphic 1-form defined near $s$ by
\[
\mathcal{N}_{f}:=\text{dlog}\left(\frac{\text{d}\left(\varphi_{2}\circ f\right)}{\text{d}\varphi_{1}}\right),
\]
where $\text{dlog}\left(g\right):=\text{d}g/g$ for any holomorphic function $g$ with non-vanishing differential. The following Proposition gives the expression of the non-linearity of a composition of two functions. Its proof is elementary and can be found in \cite{BuffRaissy}.

\begin{prop}
\label{prop:chain}
    Let $\mathcal{S}_{1}$, $\mathcal{S}_{2}$ and $\mathcal{S}_{3}$ be two affine structures, and $f:\mathcal{S}_{1}\rightarrow \mathcal{S}_{2}$ and $g:\mathcal{S}_{2}\rightarrow \mathcal{S}_{3}$ be two holomorphic functions with non-vanishing differentials. Then, $\mathcal{N}_{g\circ f}=\mathcal{N}_{f}+f^{*}\mathcal{N}_{g}$.
\end{prop}

Now, let $\mathbb{S}$ be a compact Riemann surface and $\Sigma\subset\mathbb{S}$ be a finite set. Let also $\mathcal{S}$ be an affine structure on $\mathbb{S}\setminus\Sigma$. Note that if $\varphi_{1}$ and $\varphi_{2}$ are two holomorphic coordinates defined on some open subset $U\subset\mathbb{S}$, then the non-linearities of $\varphi_{1}:\mathcal{S}\cap U\rightarrow\mathbf{C}$ and $\varphi_{2}:\mathcal{S}\cap U\rightarrow\mathbf{C}$ differ only by a holomorphic 1-form. Therefore, the polar parts of $\mathcal{N}_{\varphi_{1}}$ and $\mathcal{N}_{\varphi_{2}}$ coincide along $U\cap\Sigma$. Any time we are concerned only by the polar part of the non-linearities, we will, through misuse of language, talk about ``the" non-linearity.

\begin{definition}
    A meromorphic affine structure on a compact Riemann surface $\mathbb{S}$ is an affine structure on $\mathbb{S}\setminus\Sigma$, where $\Sigma$ is a finite set of points; moreover, we require that the non-linearity of any holomorphic chart $\varphi:U\rightarrow\mathbf{C}$, where $U\subset\mathbb{S}$, has at worst poles at the points of $U\cap S$ (called singularities).
\end{definition}

The following Proposition shows that the ``type" of the surface can be recovered from the residues of the non-linearity.

\begin{prop}
    Let $\mathcal{S}$ be a meromorphic affine structure. Then,
    \begin{enumerate}
        \item $\mathcal{S}$ is a dilation surface if and only if for any holomorphic chart $z$, the real part of all the residues of $\mathcal{N}_{z}$ are integers;
        \item $\mathcal{S}$ is flat if and only if for any holomorphic chart $z$, the real part of all the residues of $\mathcal{N}_{z}$ are real.
    \end{enumerate}
\end{prop}

\begin{definition}
    Let $\mathbb{S}_{1},\mathbb{S}_{2}$ be two compact Riemann surfaces and $\mathcal{S}_{1}$ and $\mathcal{S}_{2}$ two meromorphic affine structures on $\mathbb{S}_{1}$ and $\mathbb{S}_{2}$ respectively. We say $\mathcal{S}_{1}$ and $\mathcal{S}_{2}$ are equivalent if there exists a biholomorphism $f:\mathbb{S}_{1}\rightarrow\mathbb{S}_{2}$ such that $f:\mathcal{S}_{1}\rightarrow\mathcal{S}_{2}$ is affine.
\end{definition}

\begin{lemma}
\label{lem:aff_parti_pol}
    Let $\mathbb{S}_{1},\mathbb{S}_{2},\mathcal{S}_{1}$ and $\mathcal{S}_{2}$ be as above, and let $f:\mathbb{S}_{1}\rightarrow\mathbb{S}_{2}$ be a biholomorphism. Then $\mathcal{N}_{f}$ is a holomorphic 1-form on $\mathbb{S}_{1}$ if and only if for any $z\in\mathbb{S}_{1}$ and any holomorphic charts $\varphi_{1},\varphi_{2}$ around $z$ and $f(z)$ respectively, the polar part of $\mathcal{N}_{\varphi_{1}}$ is the pull-back by $f$ of the polar part of $\mathcal{N}_{\varphi_{2}}$.
\end{lemma}

\begin{proof}
    From what precedes, $\mathcal{N}_{\varphi_{2}\circ f}-\mathcal{N}_{\varphi_{1}}$ is a holomorphic 1-form around $z$. Moreover, by Lemma \ref{prop:chain}, $\mathcal{N}_{\varphi_{2}\circ f}=\mathcal{N}_{f}+f^{*}\mathcal{N}_{\varphi_{2}}$. Therefore $f^{*}\mathcal{N}_{\varphi_{2}}-\mathcal{N}_{\varphi_{1}}$ is holomorphic if and only if $\mathcal{N}_{f}$ is holomorphic around $z$. 
\end{proof}

This Lemma is of particular interest when $\mathbb{S}_{1}$ and $\mathbb{S}_{2}$ are the Riemann sphere, because there the only global holomorphic 1-form is the zero 1-form. Especially, since a Möbius transform is determined by the image of three points, a meromorphic affine structure on $\widehat{\mathbf{C}}$ with at most three singularities is in some sense ``determined" by the polar parts of the non-linearities around the singularities.

\subsection{From polygons to meromorphic affine structures on $\widehat{\mathbf{C}}$}

The behavior of the geodesic flow on the Riemann sphere is not well understood for a general meromorphic affine structure. That is why, in this paper, we focus on a specific type of affine structure. More precisely, let $\mathcal{Q}$ be a quadrilateral with vertices $A$, $B$, $C$ and $D$ as in Figure \ref{fig:Quadrilatere}, and let $\mathbb{S}$ be the topological sphere obtained by gluing $\left[A,B\right]$ with $\left[B,C\right]$ and $\left[C,D\right]$ with $\left[A,D\right]$ by affine maps preserving the orientation and fixing $B$ and $D$ respectively. These gluings induce a meromorphic affine structure $\mathcal{S}$ on $\mathbb{S}$, while the standard complex structure on $\mathbb{S}$ identifies it with $\widehat{\mathbf{C}}$. It is clear that the singularities of $\mathcal{S}$ correspond to the vertices of the quadrilateral (note that $A$ and $C$ are identified). These correspond to three points in $\widehat{\mathbf{C}}$, that we may choose so that $A$ (and $C$), $B$ and $D$ correspond to $0,1$, and $\infty$ respectively. In the following we denote by $\Phi:\widehat{\mathbf{C}}\longrightarrow\mathcal{S}$ the identity map on $\mathbb{S}$ seen as a map between meromorphic affine structures.

\begin{figure}[ht]
\centering
\includegraphics[scale=0.25]{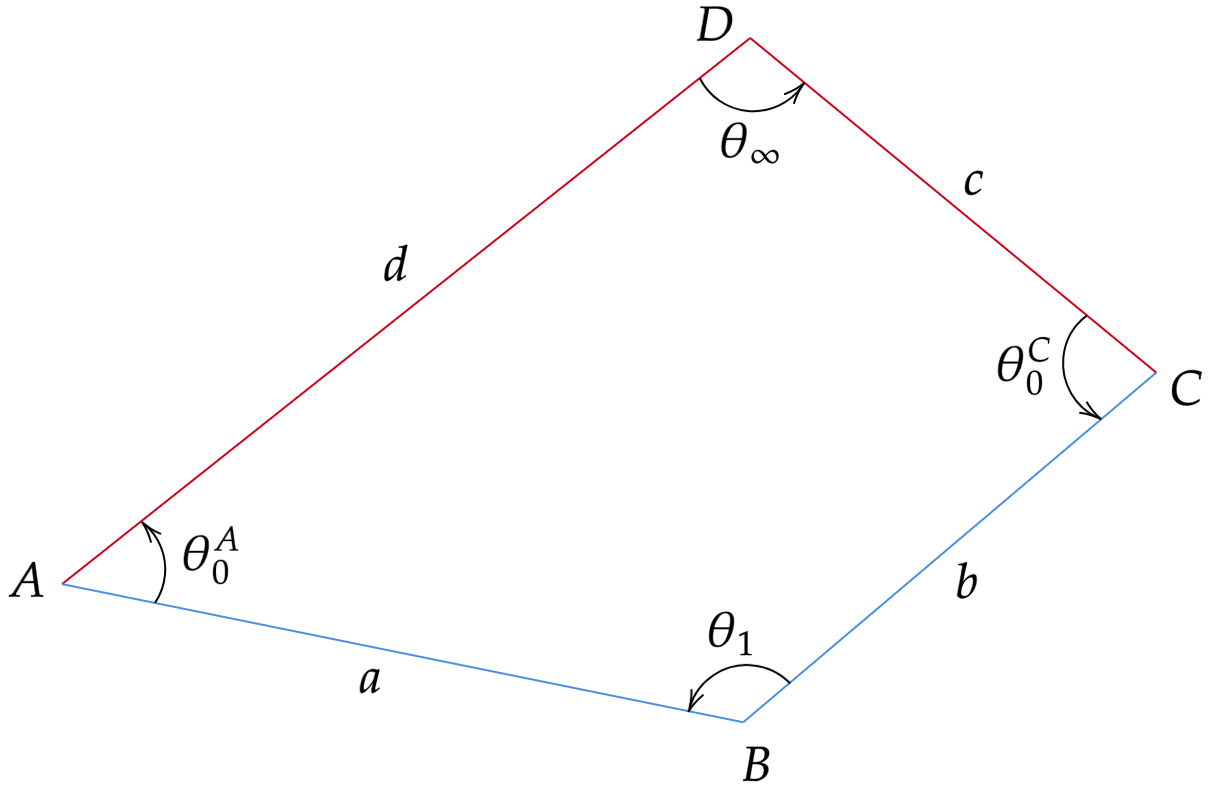}
\caption{\centering An affine surface obtained by gluing edges in a quadrilateral. The angles are measured in radians, and the letters $a,b,c$ and $d$ stand for the lengths of the segments}
\label{fig:Quadrilatere}
\end{figure}

Each singularity of $\mathcal{S}$ is encoded by the angle and the scaling ratio at the corresponding vertex. More precisely, using the notations of Figure 1, we define the following three numbers
\[
\alpha_{0}:=\theta_{0}+i\log\frac{bd}{ac},\quad\alpha_{1}:=\theta_{1}+i\log\frac{a}{b},\quad\text{and}\quad\alpha_{\infty}:=\theta_{\infty}+i\log\frac{c}{d},
\]
where $\theta_{0}:=\theta_{0}^{A}+\theta_{0}^{C}$. For convenience, we also introduce the numbers $\mu_{\zeta}:=\alpha_{\zeta}/2\pi$. On $\widehat{\mathbf{C}}$ we have two coordinates, $z$ (around 0) and $1/z$ (around $\infty$). It turns out that for $\zeta\in\{0,1\}$, the residue of $\mathcal{N}_{z}$ at $\zeta$, denoted $\text{res}_{\zeta}$, is given by
\[
\text{res}_{\zeta}=1-\frac{\alpha_{\zeta}}{2\pi}=1-\mu_{\zeta}.
\]
In fact we have
\[
\mathcal{N}_{z}=\left(\frac{1-\mu_{0}}{z}+\frac{1-\mu_{1}}{z-1}\right)\text{d}z.
\]
An easy computation shows that the residue of $\mathcal{N}_{1/z}$ at $\infty$ is exactly $1-\mu_{\infty}$. Finally, one can check that in charts around $\zeta\in\{0,1,\infty\}$, $\Phi(w)=w^{\mu_{\zeta}}(1+O(w))$.

\subsection{The affine surface associated to a given homogeneous vector field}

Let's go back to homogeneous vector fields in $\mathbf{C}^{2}$. If $\mathbf{v}=v_{1}\left(x,y\right)\partial_{x}+v_{2}\left(x,y\right)\partial_{y}$ is homogeneous of degree $k$, we define $p\circ g\left(x,y\right):=P\left(x,y\right)/y^{k+1}$ and $F\circ g\left(x,y\right):=v_{1}\left(x,y\right)/v_{2}\left(x,y\right)$, where $g(x,y)=h\circ\pi(x,y)=x/y$. Then, following again \cite{BuffRaissy}, the non-linearity of the chart $z$ on $\widehat{\mathbf
C}$ with respect to the meromorphic affine surface induced by $\mathbf{v}$ is given by

\[
\mathcal{N}_{z}=\left(\frac{p'\left(z\right)}{p\left(z\right)}-\frac{k-1}{z-F\left(z\right)}\right)\text{d}z.
\]
In the following, we assume $\mathbf{v}$ is of degree one and has three distinct characteristic directions. Without loss of generality, we can assume these are $L_{0}:=\left\{x=0\right\}$, $L_{1}:=\left\{x=y\right\}$ and $L_{\infty}:=\left\{y=0\right\}$. Then $\mathbf{v}$ can be written in the following form:
\[
\left\{
\begin{array}{l}
v_{1}\left(x,y\right)=-\alpha_{\infty}x^{2}+\left(\alpha_{1}+\alpha_{\infty}\right)xy\\[3pt]
v_{2}\left(x,y\right)=\left(\alpha_{0}+\alpha_{1}\right)xy-\alpha_{0}y^{2}
\end{array}
\right.
.
\]
This form is unique, up to permutation of the characteristic directions, if we impose the normalization $\sum\alpha_{\zeta}=2\pi$. Note that on each characteristic direction, $\mathbf{v}$ reduces to a quadratic vector field in $\mathbf{C}$. For instance, on $L_{0}$, the corresponding vector field in $\mathbf{C}$ is $-\alpha_{0}z^{2}\partial_{z}$. Its trajectories are given by $c(t)=1/(w+\alpha_{0}t)$ with $w\in\widehat{\mathbf{C}}$ (see Figure \ref{fig:Trajectoires_az2}).

\begin{figure}[ht]
\centering
\includegraphics[scale=0.25]{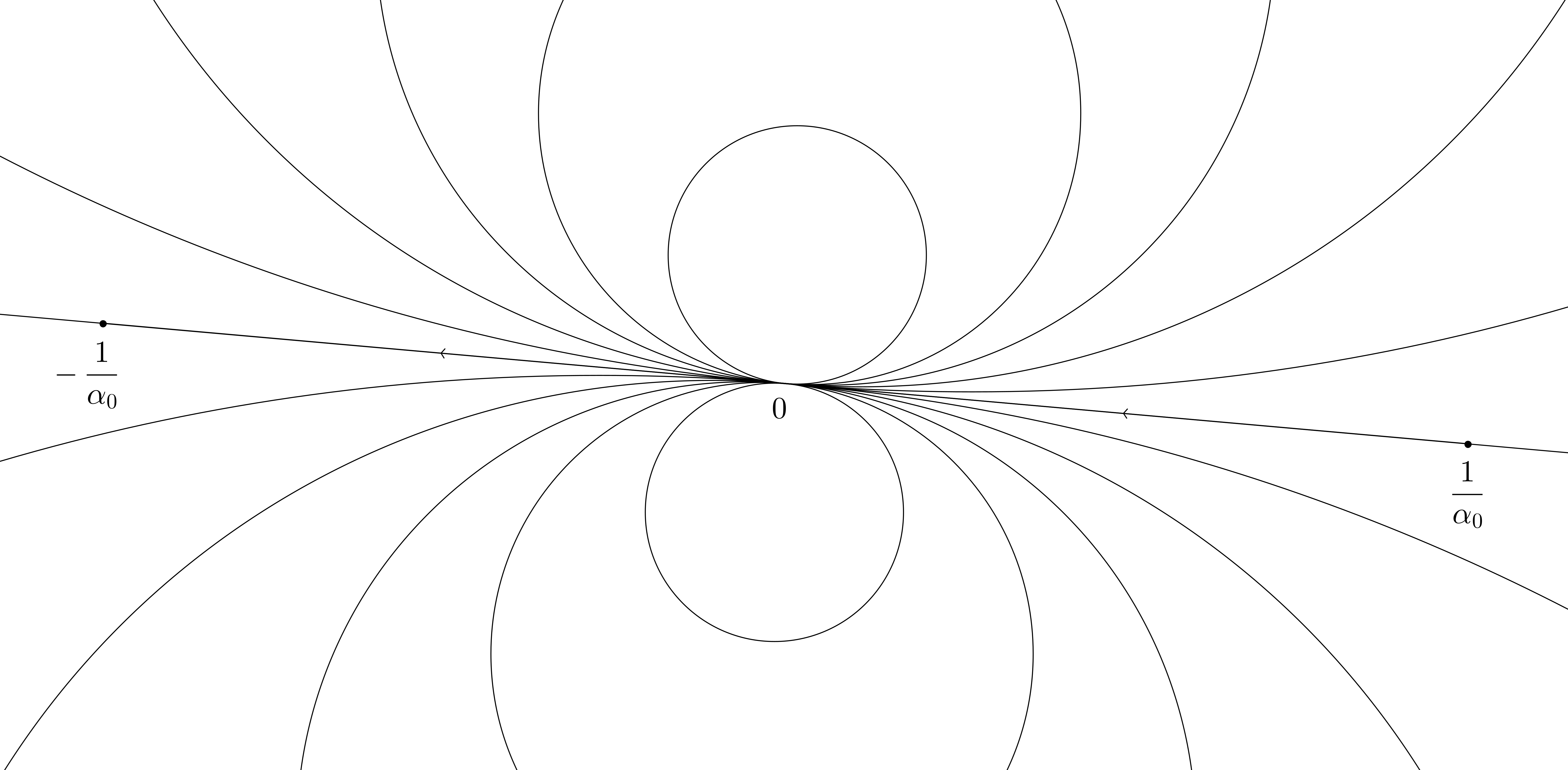}
\caption{\centering Some trajectories of the vector field $-\alpha_{0}z^{2}\partial_{z}$ for $\alpha_{0}=5\pi/4+(i\log2)/2$. As $t$ increases, the trajectories turn clockwise above the line $\left\{-t/\alpha_{0},t\in\mathbf{R}\right\}$, and counter-clockwise below}
\label{fig:Trajectoires_az2}
\end{figure}

With the normalization imposed, we get $P\left(x,y\right)=2\pi xy(x-y)$, hence $p(z)=2\pi z(z-1)$. Moreover, a simple computation shows that
\[
\mathcal{N}_{z}=\left(\frac{1-\mu_{0}}{z}+\frac{1-\mu_{1}}{z-1}\right)\text{d}z,
\]
where we have set $\mu_{\zeta}:=\alpha_{\zeta}/2\pi$ as before. Therefore, if the $\mu_{i}$'s come from a quadrilateral $\mathcal{Q}$ as explained before, then Lemma \ref{lem:aff_parti_pol} implies that the identity map from $\mathcal{S}_{\mathbf{v}}$ to $\mathcal{S}$ is affine. As a consequence, geodesics of $\mathcal{S}_{\mathbf{v}}$ correspond to geodesics of $\mathcal{S}$, and this allows to study them in the quadrilateral $\mathcal{Q}$, where they become (piecewise) straight.

\subsection{Main results}

Even if the affine structure comes from a quadrilateral as before, the geodesic flow can be tremendously difficult to understand. That is why we choose a very specific quadrilateral. From now on, let $\mathcal{Q}$ be the quadrilateral $ABCD$, where the points $A,B,C$ and $D$ have affixes $0,-i,2+i$ and $i$, respectively (see Figure \ref{fig:Triangle}). We use the same gluings and notations as previously, so that $\mathcal{S}$ is the resulting affine surface. We have 
\[
\alpha_{0}=\frac{5\pi}{4}+\frac{i}{2}\log2,\quad\alpha_{1}=\frac{\pi}{4}-\frac{3i}{2}\log2,\quad\text{and}\quad\alpha_{\infty}=\frac{\pi}{2}+i\log2.
\]

\begin{figure}[ht]
\centering
\includegraphics[scale=0.25]{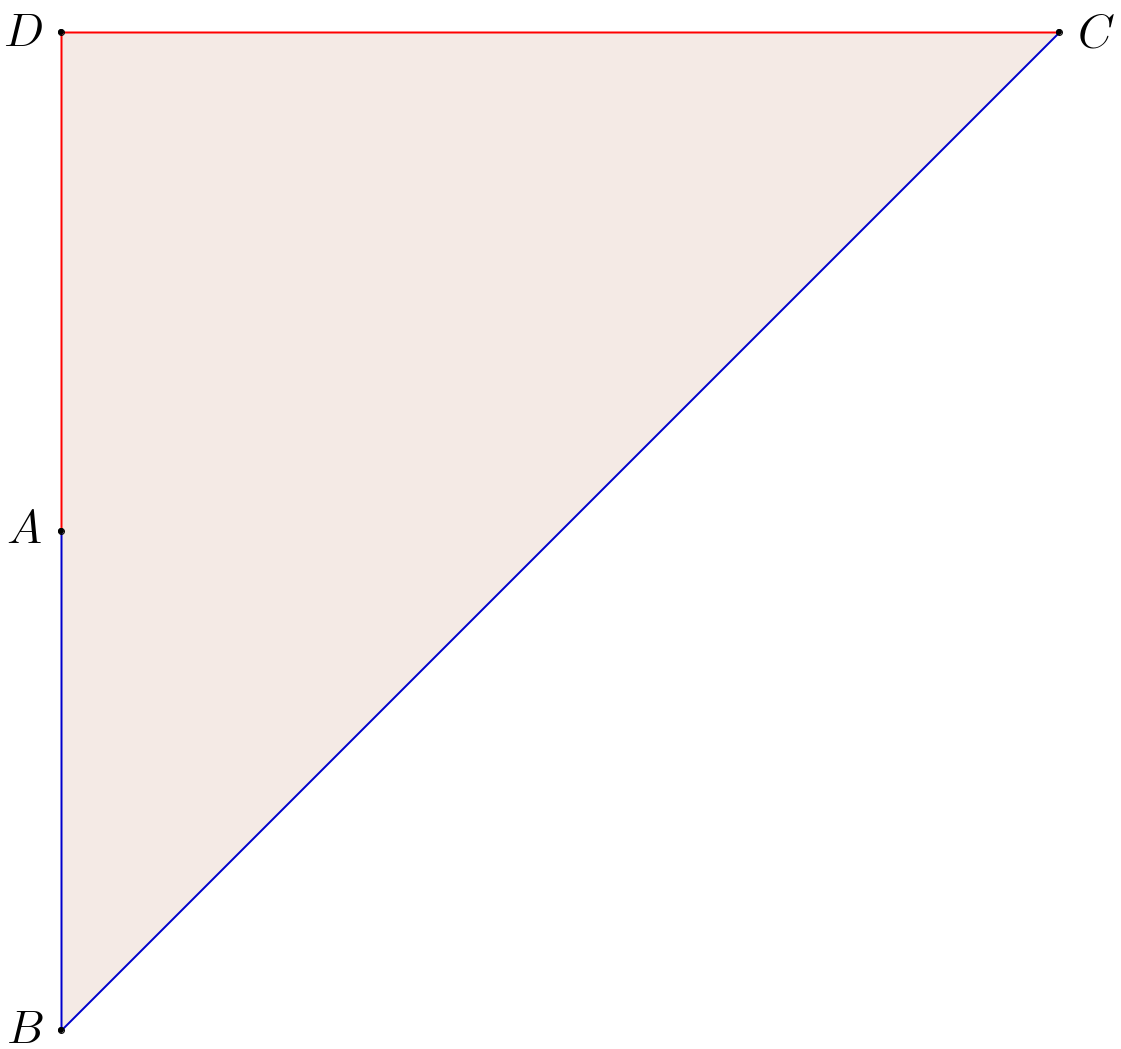}
\caption{A polygonal model of the surface $\mathcal{S}$}
\label{fig:Triangle}
\end{figure}

Let us note that this surface exhibits very similar behaviors to those of the Disco surface studied in \cite{BoulangerFougeronGhazouani}. More precisely, they show that the study of the geodesic flow in the Disco surface reduces to the study of a special kind of affine interval exchange transformations. These transformations will appear crucially in the study of $\mathcal{S}$. Nonetheless, the surface $\mathcal{S}$ is of particular interest because it induces a homogeneous vector field in $\mathbf{C}^{2}$ that we denote by $\mathbf{v}$. From what precedes, it can be written
\[
\left\{
\begin{array}{l}
\displaystyle v_{1}\left(x,y\right) = -\left(\frac{\pi}{2}+i\log2\right)x^{2}+\left(\frac{3\pi}{4}-\frac{i}{2}\log2\right)xy\\[3pt]
\displaystyle v_{2}\left(x,y\right) = \left(\frac{3\pi}{2}-i\log2\right)xy-\left(\frac{5\pi}{4}+\frac{i}{2}\log2\right)y^{2}
\end{array}
\right.
.
\]

Its trajectories for $\mathbf{v}$ will be denoted by $\gamma=\left(\gamma_{1},\gamma_{2}\right)$, and their projection under $\pi$ by $\delta:=\gamma_{1}/\gamma_{2}$.

\begin{comment}

\end{comment}

\begin{definition}
    We say that a trajectory for $\mathbf{v}$ is ``regular" if it is defined for infinite forward time, and ``irregular" otherwise.
\end{definition}

As a consequence of the classical ``Blow-up in finite time Theorem", all irregular trajectories for $\mathbf{v}$ blow-up in finite time. The following Proposition will follow easily from the study of the geodesics in $\mathcal{S}$.

\begin{prop}
\label{prop:accumulation}
    The accumulation set of any regular trajectory contains the origin.
\end{prop}

We finally state the main result of this paper. The proof will be given in Section \ref{sec:Chp_vect}.

\begin{thm}
\label{thm:main}
    We have the following decomposition of $\mathbf{C}^{2}\setminus\mathcal{C}$, invariant by $\mathbf{v}$:
    \begin{itemize}
        \item an open set $\mathcal{U}$ of full measure\footnote{By this we mean that the complement $\mathbf{C}^{2}\setminus\mathcal{U}$ has zero Lebesgue measure.}, with an infinite number of connected components. Inside $\mathcal{U}$, the regular trajectories tend to the origin.
        \item If $U$ is a connected component of $\mathcal{U}$, then the boundary $\partial U$ of $U$ in $\mathbf{C}^{2}\setminus\mathcal{C}$ consists of two components. The accumulation set of a regular trajectory starting from $\partial U$ is one of the circular trajectories contained in $L_{0}$.
        \item The other components of $\mathbf{C}^{2}\setminus(\mathcal{C}\cup\mathcal{U})$ are real manifolds of dimension 3. The accumulation set of each regular trajectory starting from one of these components is the union of two or three trajectories (non-reduced to the origin) in $L_{0}$.
    \end{itemize}
\end{thm}

In fact, we will make this statement more accurate. In particular, we will define a "rotation number" on $\mathbf{C}^{2}\setminus\mathcal{C}$, and we will show that each connected component $U$ of $\mathcal{U}$ corresponds to a rational value of the rotation number. Moreover, for any trajectory $\gamma$ starting from $U$, its projection $\delta$ accumulates a closed geodesic, whose homotopy class in $\mathbf{C}\setminus\left\{0,1\right\}$ completely determines $U$.

\vspace{1\baselineskip}

\textbf{Acknowledgments.} 

I would like to thank Xavier Buff for introducing me to these questions and for carefully reviewing the preliminary versions of this paper. This research was greatly inspired by discussions during the thematic semester \textit{Holomorphic Dynamics and Geometry of Surfaces}, held from February to June 2025 and funded by the LabEx CIMI. I am also grateful to Selim Ghazouani for very fruitful discussions.

\section{Geodesics in $\mathcal{S}$}

\subsection{Reduction to a one-dimensional dynamical system}

In this section we investigate the geodesic flow on $\mathcal{S}$. More precisely, we study geodesics in the fundamental domain $\mathcal{Q}$. Since the angles in $\mathcal{Q}$ are rational, the dynamics on $\mathcal{S}$ is very similar to the dynamics on a dilation surface. In fact, it is covered by the dilation surface $\widetilde{\mathcal{S}}$ pictured in Figure \ref{fig:Revetement}. The problem is that $\widetilde{\mathcal{S}}$ has genus 3, and the covering has degree 8. So even though in $\widetilde{\mathcal{S}}$ the direction of the geodesics is preserved, we will only work in $\mathcal{S}$.

\begin{figure}[ht]
\centering
\includegraphics[scale=0.4]{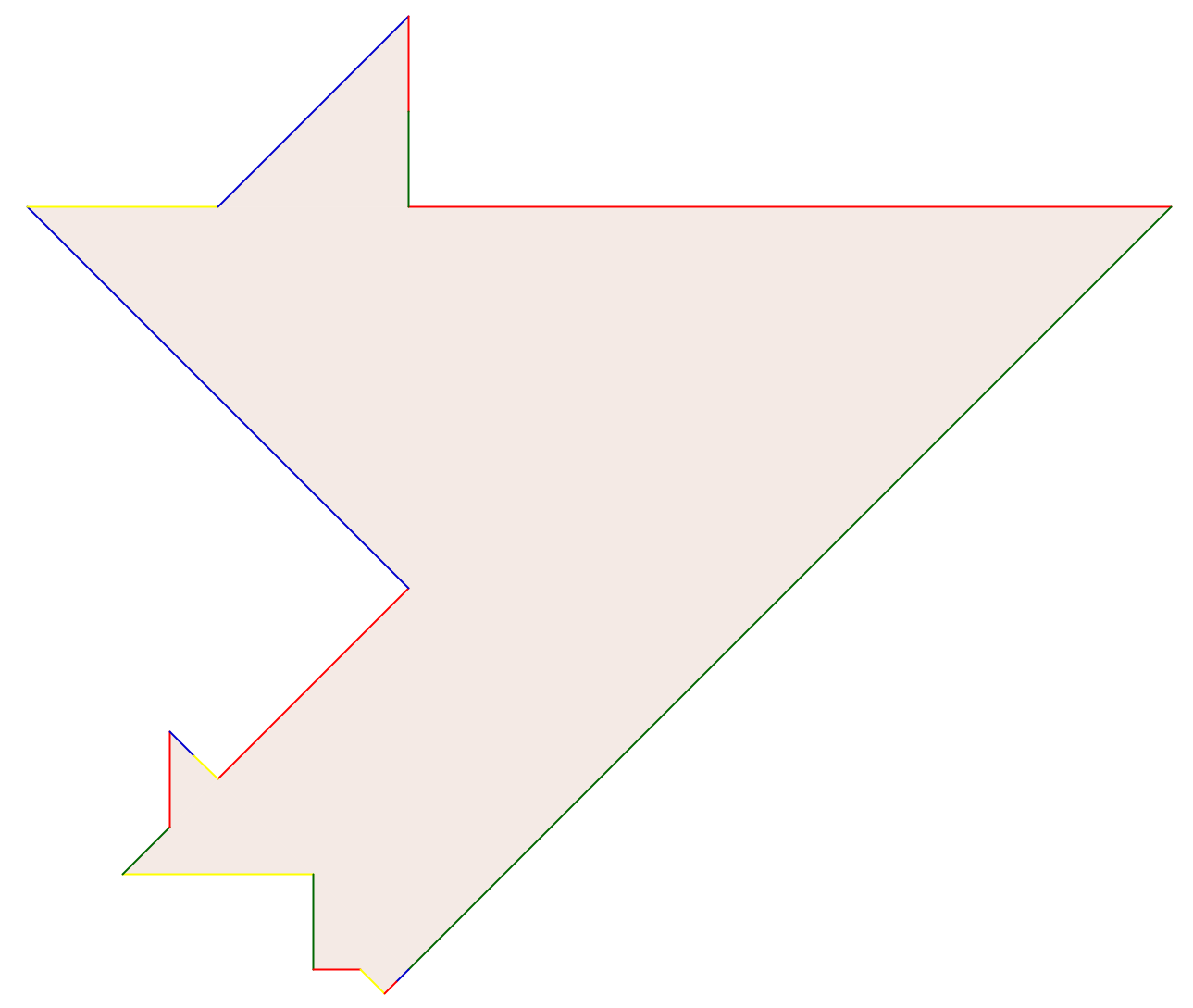}
\caption{\centering A degree 8 covering of $\mathcal{S}$ by a genus 3 dilation surface. Parallel edges of the same color are identified}
\label{fig:Revetement}
\end{figure}

%Another example of a meromorphic affine sphere that could be studied in the same way, denoted by $\mathcal{S'}$, is defined similarly by gluing $\left[A',B'\right]$ with $\left[B',C'\right]$ and $\left[C',D'\right]$ with $\left[A',D'\right]$ in the quadrilateral $A'B'C'D'$, where the points $A',B',C'$ and $D'$ have coordinates $\left(0,0\right),\left(1,0\right),\left(0,1\right)$ and $\left(-1,0\right)$, respectively \textcolor{blue}{(ajouter une image)}.\\
Because geodesics in $\mathcal{Q}$ are piecewise straight, we can talk the angle of the geodesic at any given point. We will measure all the angles with respect to the positive real half-line. Also, let $\Omega$ be the set of points $(x,\theta)\in\partial\mathcal{Q}\times\left[0,2\pi\right]$ such that the geodesic emanating from $x$ with initial angle $\theta$ enters $\mathcal{Q}$. We shall explain how the geodesic flow on the surface reduces to a discrete dynamical system $T$ on $\Omega$. So let $(x,\theta)\in\Omega$ and $\delta$ be the geodesic emanating from $x$ with angle $\theta$. Let $y$ be the next point of $\partial\mathcal{Q}$ reached by $\delta$. If $y$ is a singularity, then $T$ is not defined at $(x,\theta)$. Otherwise, $\delta$ will leave $\partial\mathcal{Q}$ from a new point $z$ (which is identified with $y$ in $\mathcal{S}$) with an angle $\theta'$. We then define $T(x,\theta):=(z,\theta')$. The set $\Omega$ is called the \textit{phase space} of $T$ or the \textit{Poincaré section} of the geodesic flow. In the following, we will write $\left(x_{n},\theta_{n}\right)$ for the $n$-th iterate (if defined) of a point $\left(x_{0},\theta_{0}\right)$ under $T$. Following this point of view, a geodesic for $\mathcal{S}$ corresponds to an orbit under $T$, uniquely determined by its initial position on the boundary of the quadrilateral and its initial angle.\\
In the following, we will use the standard Euclidean metric on $\mathcal{Q}$ to measure lengths of tangent vectors.
Moreover, geodesics can be time-parameterized, and we may assume that the initial speed is always $1$. A classical way of studying a given geodesic is to unfold the triangle along this geodesic, so that it becomes a straight line, hence traveling at constant speed with respect to the Euclidean metric on the plane. We begin with the following trivial, but important, remark.

\begin{rmk}
    From Figure \ref{fig:Triangle}, it is clear that if the initial point $x_{0}\in\left]A,D\right[\cup\left]A,B\right[$, then, for all $n\in\mathbf{N}$, $x_{n}\in\left]A,D\right[\cup\left]A,B\right[$, provided it is well defined.
\end{rmk}

By analogy with the trajectories of $\mathbf{v}$, we say a geodesic is ``regular" if it is defined for infinite forward time.

\begin{lemma}
\label{RestrictABAD}
    Any regular geodesic hits after a finite number of iterations one of the edges $\left]B,C\right[$ and $\left]C,D\right[$.
\end{lemma}

\begin{proof}
    Consider a geodesic that only hits the edges $\left]A,B\right[$ and $\left]A,D\right[$. Unfold the initial triangle along the geodesic. Each new triangle gets contracted by a factor of at least $2$. Since the initial triangle is enclosed in a ball of diameter $2\sqrt{2}$, the entire geodesic remains contained within a ball of diameter $2\sqrt{2}\sum1/2^{n}=4\sqrt{2}$. This means that the geodesic travels a finite distance, but since it travels at constant speed, it cannot be regular.
\end{proof}

 We will focus on regular geodesics. From what precedes, the speed of a regular trajectory tends to 0. According to Lemma \ref{RestrictABAD}, we can assume that our trajectories begin in the intervals $\left]A,B\right[$ or $\left]A,D\right[$. Additionally, after possibly one iteration, we can assume that the initial angle is within the range $\left[0,\pi/4\right]$. Finally, we will utilize the following lemma to further limit our focus to trajectories that start at the edge $\left]A,B\right[$.

\begin{lemma}
\label{RestrictAB}
    Let $x_{0}\in\left]A,D\right[$ and $\theta_{0}\in\left[0,\pi/4\right]$. Assume $\left(x_{3},\theta_{3}\right)$ is well defined. Then $\theta_{3}=\theta_{0}$ and $x_{3}\in\left]A,B\right[$.
\end{lemma}

\begin{proof}
    If $x_{1}\in\left]A,B\right[$, then $\theta_{1}=\theta_{0}+\pi/4\in\left[\pi/4,\pi/2\right]$, which implies $x_{2}\in\left]A,D\right[$ and $\theta_{2}=\theta_{1}-\pi/2\in\left[-\pi/4,0\right]$. Therefore, $x_{3}\in\left]A,B\right[$ and $\theta_{3}=\theta_{2}+\pi/4=\theta_{0}$. Similarly, if $x_{1}\in\left]A,D\right[$, then $\theta_{1}=\theta_{0}-\pi/2\in\left[-\pi/2,-\pi/4\right]$, which implies $x_{2}\in\left]A,B\right[$ and $\theta_{2}=\theta_{1}+\pi/4\in\left[-\pi/4,0\right]$. Therefore, $x_{3}\in\left]A,B\right[$ and $\theta_{3}=\theta_{2}+\pi/4=\theta_{0}$.
\end{proof}

Henceforth, $\left(x_{0},\theta_{0}\right)$ is a point in $\left[A,B\right]\times\left[0,\pi/4\right]$ whose forward orbit under $T$ is well defined. From the proof of Lemma \ref{RestrictAB}, we have $\theta_{3}=\theta_{0}$ and $\theta_{1},\theta_{2}\neq\theta_{0}$. The first return map to the angle $\theta_{0}$ is therefore defined by $T_{\theta_{0}}\left(x_{0}\right):=x_{3}$. The upshot is that we have reduced the dynamics on the surface to a dynamics on an interval. Therefore, we may write the initial angle $\theta$ instead of $\theta_{0}$. Proposition \ref{prop:echangesintervinduits} gives a precise description of the induced map. For this we need the following definition.

\begin{definition}
\label{def:gAIET}
    Let $I=[a,b]$ be a non-trivial interval. Let $n\geq2$ be an integer and $x_{1},\dots,x_{n+1}\in I$ such that $a=x_{1}<\dots<x_{n+1}=b$. For $i=1,\dots,n$, let $f_{i}:[x_{i},x_{i+1}]\rightarrow I$ be an affine map that preserves the orientation. Let us define $f:I\setminus\{x_{2},\dots,x_{n}\}\rightarrow I$ by $f(x)=f_{i}(x)$ if $x\in]x_{i},x_{i+1}[$ and by $f(a)=f_{1}(a)$ and $f(b)=f_{n}(b)$. If $f$ is injective, then it is called an interval exchange transformation with gaps (g-AIET).\\
    A contracting g-AIET is a g-AIET $f$ such that all $f_{i}$'s are contracting.
\end{definition}

For $i=2,\dots,n$, each $x_{i}$ in Definition \ref{def:gAIET} is called a ``singularity" of $f$. Note that through each singularity $x_{i}$, $f$ extends continuously to the left by setting $f(x_{i})=f_{i-1}(x_{i})$, and to the right by setting $f(x_{i})=f_{i}(x_{i})$.

\begin{prop}
    \label{prop:echangesintervinduits}
    Writing $\widetilde{\theta}:=\pi/2-\arctan\left(2\right)\approx26.57^{\circ}$, we get:
    \begin{itemize}
        \item if $\theta\in\bigl[0,\widetilde{\theta}\bigr]\cup\left\{\pi/4\right\}$, then $T_{\theta}$ is a strictly contracting affine map;
        \item if $\theta\in\bigl]\widetilde{\theta},\pi/4\bigr[$, then $T_{\theta}$ is a g-AIET with two intervals (only one singularity).
    \end{itemize}
\end{prop}

\begin{proof}
    For all $\theta\in\bigl[0,\widetilde{\theta}\bigr]$ (respectively $\theta=\pi/4$) and all $x_{0}\in\left]A,B\right[$, $x_{1}\in\left]A,B\right[$ (respectively $x_{1}\in\left]A,D\right[$), because $\widetilde{\theta}$ is the angle for which we have a saddle connection joining $A$ and $C$. Therefore, no point is mapped after one iteration to a singularity, and it is not difficult to check that $T_{\theta}$ extends continuously to the whole interval $\left[A,B\right]$. Then it is easy to see that $T_{\theta}$ is a composition of strictly contracting affine maps.\\
    Now, if $\theta\in\bigl]\widetilde{\theta},\pi/4\bigr[$, then the point $\left(s,\theta\right)$, where $s$ has coordinates $\left(0,1-2\tan\theta\right)$ in $\mathbf{R}^{2}$, is mapped to the singular point $C$. Hence, the segments $\left]A,s\right[$ and $\left]s,B\right[$ are mapped to $\left]C,D\right[$ and $\left]B,C\right[$, respectively. At this point, it is not hard to see that $\left(s,\theta\right)$ is the only point whose image under $T_{\theta}$ is not well defined. Finally, an easy computation shows that, seeing $[A,B]$ as $[0,1]$,
    \begin{equation}
    \label{eq:formules_T_theta}
    T_{\theta}(x)=\left\{
    \begin{array}{cc}
        \displaystyle\frac{x}{16}+1-\frac{4\tan\theta-1}{16} & \text{if }x\in[0,2\tan\theta-1[ \\
        \displaystyle\frac{x-1}{16}+1-\frac{3\tan\theta+1}{4} & \text{if }x\in]2\tan\theta-1,1]
    \end{array}
    \right.
    .
    \end{equation}
    Figures \ref{fig:Triangle_deplie} and \ref{fig:Graph} are quite helpful.    
\end{proof}

\begin{figure}[ht]
\centering
\includegraphics[scale=0.25]{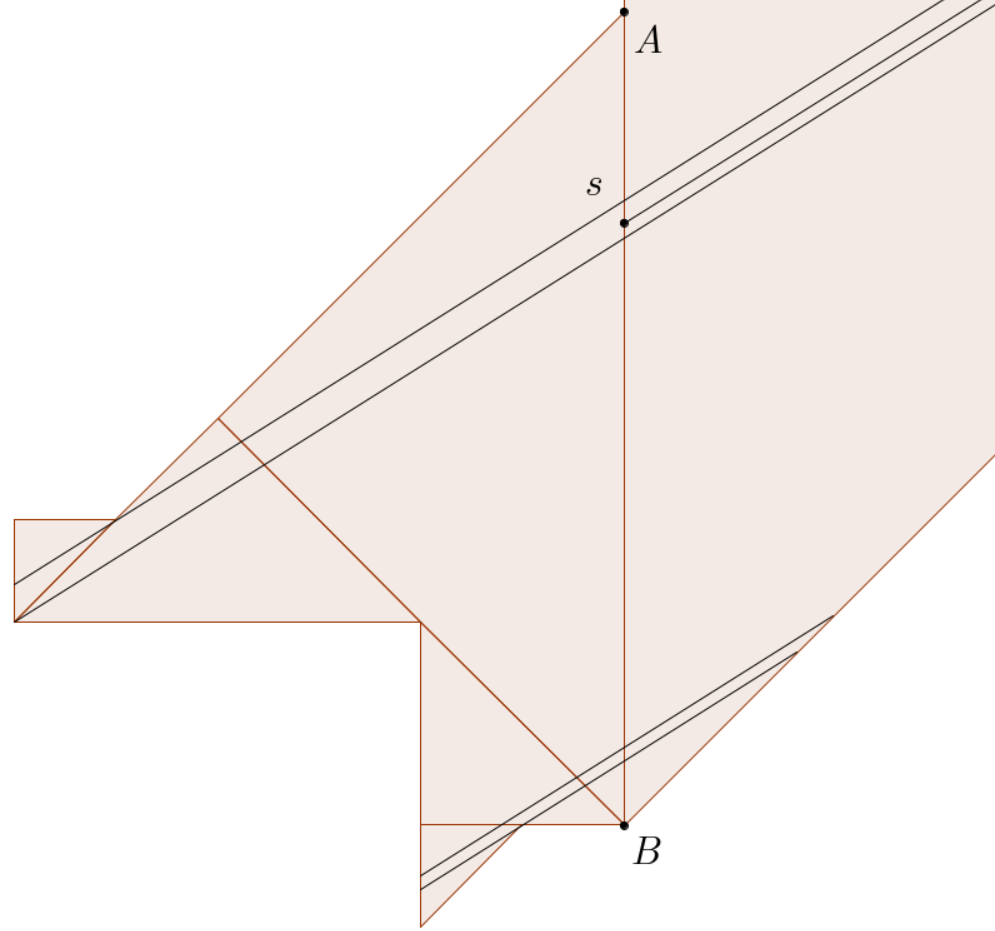}
\caption{\centering The initial quadrilateral is unfolded ``backward" to visualize the g-AIET. The point $s$ hits the singularity at $C$}
\label{fig:Triangle_deplie}
\end{figure}

\begin{cor}
\label{Limitepetitangle}
    Let $\left(x_{0},\theta_{0}\right)\in\left[A,B\right]\times\left(\bigl[0,\widetilde{\theta}\bigr]\cup\left\{\pi/4\right\}\right)$ such that $\left(x_{n}\right)_{n\in\mathbf{N}}$ is well defined. Then $\left(x_{n}\right)_{n\in\mathbf{N}}$ is attracted by a $3$-cycle.
\end{cor}

\begin{proof}
    The only non-trivial fact at this point is that the limit of $\left(x_{3n}\right)_{n\in\mathbf{N}}$ is neither $0$ nor $1$, but this is easily  checked by computing the explicit expression of $T_{\theta}$. Details are left to the reader.
\end{proof}

\subsection{Dynamics of $T_{\theta}$}

Let $f:\left[0,1\right]\rightarrow\left[0,1\right]$ be a g-AIET (see Definition \ref{def:gAIET}). Then, for any $x\in\left[0,1\right]$ whose forward orbit is well defined, its omega-limit set $\omega\left(x\right)$ (which does not depend on $x$) is either a finite set, or a Cantor set, or the whole interval $\left[0,1\right]$. In the first case, $\omega\left(x\right)$ consists of a periodic orbit. It might happen that $\omega\left(x\right)$ contains a singularity $s$, in which case $s$ is periodic for one of the two extensions of $f$ through $s$ discussed above (refer to Definition \ref{def:gAIET} and to the discussion that follows). In the last case, the orbit of $x$ is dense in $\left[0,1\right]$. This phenomenon cannot occur in our situation since our g-AIETs are contracting. We shall justify that, in our situation, the first case is generic, in the sense that for almost all $\theta$'s, $T_{\theta}$ has an attracting periodic orbit. We shall also see that the set of initial angles for which the omega-limit set is a Cantor set is itself a Cantor set of zero Hausdorff dimension.

To deal with the angles $\theta\in\bigl]\widetilde{\theta},\pi/4\bigr[$, the idea is to associate to $T_{\theta}$ a rotation number. We will be able to use classical results to argue that it depends continuously on $\theta$. Angles giving a rational rotation number will correspond to case 1, while angles giving an irrational rotation number will correspond to case 2.\\
For convenience, let us first extend $T_{\theta}$ to a right-continuous map by setting
\[
\left\{
\begin{array}{rcl}
  T_{\theta}\left(0\right) & = & \lim_{x\rightarrow0}T_{\theta}\left(x\right)\\
  T_{\theta}\left(2\tan\theta-1\right) & = & \lim_{x\rightarrow\left(2\tan\theta-1\right)^{+}}T_{\theta}\left(x\right)
\end{array}
\right.
.
\]
It then induces a discontinuous map of the circle (see Figure \ref{fig:Graph}), and we can consider the lift $\widetilde{T_{\theta}}$ defined on $\left[0,1\right[$ by
\[
    \widetilde{T_{\theta}}\left(x\right)=\left\{
    \begin{array}{cc}
        T_{\theta}\left(x\right) & \text{if }x\in\left[0,2\tan\theta-1\right[ \\
        T_{\theta}\left(x\right)+1 & \text{if }x\in\left[2\tan\theta-1,1\right[
    \end{array}
    \right.
\]
and extended to $\mathbf{R}$ by the relation 
\begin{equation}
    \label{eq:Extension_degree1}
    \widetilde{T_{\theta}}\left(x+1\right)=\widetilde{T_{\theta}}\left(x\right)+1.
\end{equation}

\begin{figure}[ht]
\centering
\includegraphics[scale=0.3]{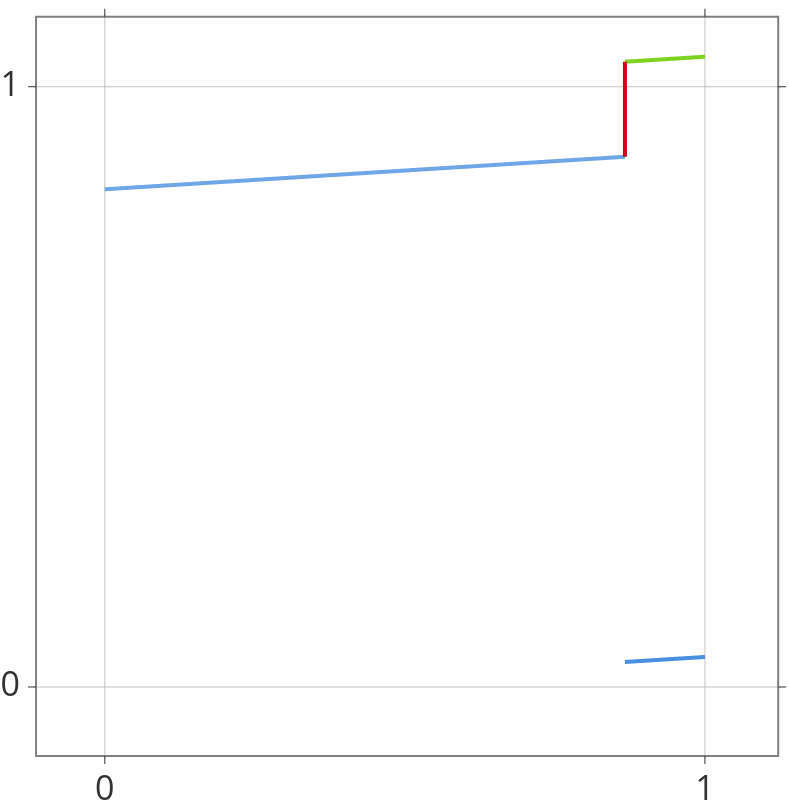}
\caption{\centering Blue: graph of $T_{\theta}$ for $\theta=\arctan\left(14/15\right)$. Green: graph of the lift $\widetilde{T_{\theta}}$. Red: extension to a set-valued map}
\label{fig:Graph}
\end{figure}

In fact, any extension of $T_{\theta}$ through the singularity lifting to an increasing function satisfying Equation \ref{eq:Extension_degree1} would work equally well. The following proposition results from \cite{RhodesThompson1}.

\begin{prop}
    For all $\theta\in\bigl[0,\pi/4\bigr]$ and all $x\in\mathbf{R}$, the limit
    \[
    \text{transl}\left(\theta\right):=\lim_{n\rightarrow+\infty}\frac{\widetilde{T_{\theta}}^{\circ n}\left(x\right)}{n}=\lim_{n\rightarrow+\infty}\frac{\widetilde{T_{\theta}}^{\circ n}\left(x\right)-x}{n}
    \]
    exits and does not depend on $x$. It is called ``translation number". Its projection on the circle $\mathbf{R}/\mathbf{Z}$ $\text{rot}(\theta):=\text{transl}\left(\theta\right)\text{mod }1$ is called ``rotation number".
\end{prop}

There are three special cases where we can compute explicitly the translation number.

\textbf{Case 1: $\text{transl}\left(\theta\right)=0$.} For $\theta>\arctan\left(16/17\right)\approx43.26^{\circ}$, $$\lim_{x\rightarrow\left(2\tan\theta-1\right)^{-}}T_{\theta}\left(x\right)<2\tan\theta-1,$$ hence the entire interval $\left]0,1\right[$ is mapped into $\left]0,2\tan\theta-1\right[$. Therefore, $\widetilde{T_{\theta}}$ has a fixed point, whence its translation number is zero. The same happens for $\theta\in\bigl[0,\widetilde{\theta}\bigr]\cup\left\{\pi/4\right\}$.

\vspace{1\baselineskip}

\textbf{Case 2: $\text{transl}\left(\theta\right)=1$.} For $\widetilde{\theta}<\theta<\arctan\left(13/21\right)\approx31.76^{\circ}$, $$\lim_{x\rightarrow\left(2\tan\theta-1\right)^{+}}T_{\theta}\left(x\right)>2\tan\theta-1,$$ hence the entire interval $\left]0,1\right[$ is mapped into $\left]2\tan\theta-1,1\right[$. Therefore, $T_{\theta}$ admits a fixed point $\widehat{x}\in\left[2\tan\theta-1,1\right[$, and for all $n\in\mathbf{N}$, $$\widetilde{T_{\theta}}^{\circ n}\left(\widehat{x}\right)=\widehat{T_{\theta}}^{\circ n}\left(\widehat{x}\right)+n=\widehat{x}+n.$$ This proves $\text{transl}\left(\theta\right)=1$.

\vspace{1\baselineskip}

\textbf{Case 3: $\text{transl}\left(\theta\right)=1/2$.} For $\arctan\left(7/11\right)<\theta<\arctan\left(11/12\right)$ i.e. approximately $32.47^{\circ}<\theta<42.51^{\circ}$, $$\lim_{x\rightarrow0^{+}}T_{\theta}\left(x\right)>2\tan\theta-1$$ and $$\lim_{x\rightarrow1^{-}}T_{\theta}\left(x\right)<2\tan\theta-1.$$ Thus, the two intervals $\left]0,2\tan\theta-1\right[$ and $\left]2\tan\theta-1,1\right[$ are mapped one into the other. Therefore, $T_{\theta}$ has a periodic point $x$ of period $2$. For this point, we have $\widetilde{T_{\theta}}^{\circ2}\left(x\right)=x+1$, which gives, for all $n\in\mathbf{N}$, $\widetilde{T_{\theta}}^{\circ2n}\left(x\right)=x+n$. Consequently, $$\text{transl}\left(\theta\right)=\lim_{n\rightarrow+\infty}\frac{\widetilde{T_{\theta}}^{\circ2n}\left(x\right)}{2n}=\lim_{n\rightarrow+\infty}\frac{x+n}{2n}=\frac{1}{2}.$$

\begin{figure}[ht]
\centering
\includegraphics[scale=0.3]{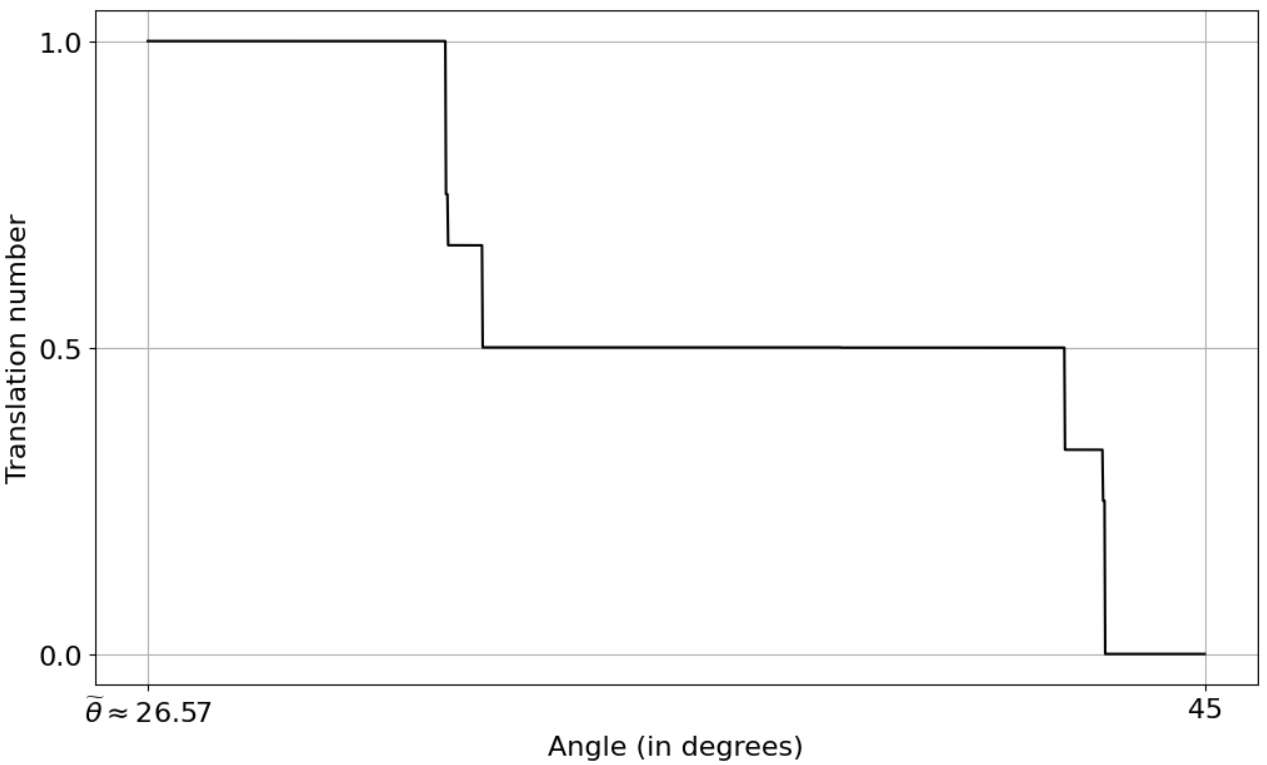}
\caption{\centering Graph of the function transl. The only rational values of the translation number that are clearly visible are $0$, $1/3$, $1/2$, $2/3$, and $1$. The values $1/4$ and $3/4$ are barely visible. This picture prefigures Proposition \ref{prop:Dim_Hausdorff}}
\label{fig:Graphe_transl_num}
\end{figure}

The following theorem describes the behavior of the geodesics in $\mathcal{S}$ according to their associated rotation number. It can be deduced from the results in \cite{RhodesThompson1} (see also \cite{Brette}). Since Corollary \ref{Limitepetitangle} describes completely the behavior of the geodesics for $\theta\in\bigl[0,\widetilde{\theta}\bigr]\cup\left\{\pi/4\right\}$, we restrict the statement to $\theta\in\bigl]\widetilde{\theta},\pi/4\bigr[$.

\begin{thm}
\label{limitedeuxinterv}
    Let $\theta\in\bigl]\widetilde{\theta},\pi/4\bigr[$ and $\text{rot}(\theta)$ be the associated rotation number. Let $\delta$ be a regular geodesic that starts at $\left]A,B\right[$ with angle $\theta$. Then,
    \begin{itemize}
        \item if $\theta$ lies in the interior of $\text{rot}^{-1}\left(\mathbf{Q}\right)$, the omega-limit set of $\delta$ is a closed geodesic;
        \item if $\theta$ lies in the boundary of $\text{rot}^{-1}\left(p/q\right)$ for some $p/q\in\mathbf{Q}$, then the omega-limit set of $\delta$ is a saddle connection;
        \item if $\text{rot}(\theta)\in\mathbf{R}\setminus\mathbf{Q}$, then the omega-limit set of any continuous lift of  $\delta$ to $\widetilde{\mathcal{S}}$ is transversely a Cantor set (locally homeomorphic to the Cartesian product of a segment with a Cantor set).
    \end{itemize}
\end{thm}

Now, we aim to justify that the function $\text{rot}$ is continuous. Since it clearly equals 0 on $\theta\in\bigl[0,\widetilde{\theta}\bigr]\cup\left\{\pi/4\right\}$ and $$\lim_{\theta\rightarrow\widetilde{\theta}^{+}}\text{rot}(\theta)=\lim_{\theta\rightarrow\pi/4^{-}}\text{rot}(\theta)=0,$$ it suffices to prove that $\text{rot}$ is continuous on $\bigl]\widetilde{\theta},\pi/4\bigr[$. Following \cite{RhodesThompson2}, we extend $x\mapsto \widetilde{T_{\theta}}\left(x\right)$ to a set-valued map by ``filling the gap" at the discontinuities (compare Figure \ref{fig:Graph}). More precisely, we set
\[
\widetilde{T_{\theta}}\left(2\tan\theta-1\right)=\left[\lim_{x\rightarrow\left(2\tan\theta-1\right)^{-}}T_{\theta}\left(x\right),\,T_{\theta}\left(2\tan\theta-1\right)+1\right].
\]
We observe that the map $x\mapsto \widetilde{T_{\theta}}\left(x\right)$ is increasing, and $\theta\mapsto \widetilde{T_{\theta}}\left(x\right)$ is decreasing, and that the map $\widetilde{T_{\theta}}$ depends continuously on $\theta$ if we put the Hausdorff topology on the graph of $x\mapsto \widetilde{T_{\theta}}\left(x\right)$. The following theorem then follows directly from the results proved in \cite{RhodesThompson2} (see also \cite{Brette}).

\begin{thm}
\label{Paramtheta}
The function $\text{rot}$ satisfies the following properties on $\bigl]\widetilde{\theta},\pi/4\bigr[$:
\begin{enumerate}
    \item it is continuous and non-increasing;
    \item for all $p/q\in\mathbf{Q}$ in the image of $\text{rot}$, $\text{rot}^{-1}\left(p/q\right)$ is a non-trivial interval;
    \item it reaches every irrational value at most once;
    \item it takes irrational values on a Cantor subset of $\bigl]\widetilde{\theta},\pi/4\bigr[$.
\end{enumerate}
\end{thm}

The upshot is that the function $\text{rot}$ is continuous from $\left[0,\pi/4\right]\text{ mod }\pi/4$ to $\left[0,1\right]\text{ mod }1$. The following Proposition says that the when $\text{rot}(\theta)\in\mathbf{R}\setminus\mathbf{Q}$, the limit Cantor set $K_{\theta}:=\bigcap_{n\in\mathbf{N}}\overline{T_{\theta}^{\circ n}([0,1]\setminus\{s\})}$ of $T_{\theta}$ is ``small".

\begin{prop}
    When $\text{rot}(\theta)\in\mathbf{R}\setminus\mathbf{Q}$, $K_{\theta}$ has zero Hausdorff dimension (and, in particular, zero Lebesgue measure).
\end{prop}

\begin{proof}
By definition, for all $n\in\mathbf{N}$, $K_{\theta}$ is covered by $T_{\theta}^{\circ n}(]0,1\setminus\{s\}[)$ up to finitely many points. But $T_{\theta}^{\circ n}(]0,1\setminus\{s\}[)$ consists of exactly $n+1$ open intervals of length at most $1/16^{n}$. Therefore, since for all $\varepsilon>0$
\[
\frac{n+1}{\left(16^{\varepsilon}\right)^{n}}\xrightarrow[n\rightarrow+\infty]{}0,
\]
the Hausdorff dimension of $K_{\theta}$ is zero.
\end{proof}

We conclude this section by proving that almost all geodesics are regular. Remember that $\Omega$ was defined as the set of $\left(x,\theta\right)$ in the phase space such that the geodesic starting from $x$ with angle $\theta$ enters the quadrilateral $\mathcal{Q}$.

\begin{lemma}
    \label{lemma:definfgen}
    For almost all $\left(x,\theta\right)\in\Omega$, the geodesic starting at $\left(x,\theta\right)$ is regular.
\end{lemma}

\begin{proof}
    First, note that an irregular geodesic either eventually hits a singularity or only hits the edges $\left]A,B\right[$ and $\left]A,D\right[$. We treat these two cases separately. In $\Omega$, the set of points corresponding to a geodesic that hits a singularity after one step forms a finite union of smooth curves. Therefore, the set of initial conditions that lead to a geodesic hitting a singularity is a countable union of smooth curves, hence a null set. Furthermore, for the same reasons, the initial conditions that lead to a geodesic hitting a singularity in the past also constitute a null set.\\
    Now, let $\Sigma\subset\Omega$ be the set of initial conditions that lead to an irregular geodesic that does not hit a singularity, either forward or backward. For a point $\left(x,\theta\right)\in\Sigma$, let $\delta$ be the geodesic emanating from $x$ with angle $\theta$, and consider the ``backward time" geodesic $\hat{\delta}:t\mapsto\delta(-t)$. By assumption, this new geodesic is regular, since it only hits the edges $\left]B,C\right[$ and $\left]C,D\right[$ without reaching a singularity. Hence, it induces a $T_{\theta}$-orbit $\left\{y_{n}\right\}_{n\in\mathbf{N}}$ on the segment $\left]A,B\right[$ for some (unique) $\theta$. But $\delta$ intersects $\left]A,B\right[$ infinitely many times, so in fact the orbit $\left\{y_{n}\right\}_{n\in\mathbf{N}}$ extends to a $T_{\theta}$-orbit $\left\{y_{n}\right\}_{n\in\mathbf{Z}}$. Consequently, the set $\left\{y_{n}\right\}_{n\in\mathbf{Z}}$ must lie within the limit set of $T_{\theta}$ since the latter coincides with $\bigcap_{n\in\mathbf{N}}\overline{T_{\theta}^{\circ n}([0,1]\setminus\{s\})}$. If $\text{rot}(\theta)\in\mathbf{Q}$, the limit set consists of finitely many points. If $\text{rot}(\theta)\in\mathbf{R}\setminus\mathbf{Q}$, the limit set forms a Cantor set of zero Lebesgue measure. This proves that the intersection of $\Sigma$ with any segment $\theta=$constant has zero measure. Finally, Fubini's theorem allows us to draw our conclusion.
\end{proof}

\begin{rmk}
    Any two regular geodesics with the same initial angle induce the same rotation number. The rotation number can therefore be continuously extended to irregular geodesics as well. As a result, the rotation number is well-defined for each initial angle, independent of the initial point.
\end{rmk} 

\subsection{Renormalization and parameter space of $T_{\theta}$}

In this section we investigate the parameter space of $\{T_{\theta}\}$. We are particularly concerned about the subset $\mathcal{K}_{\Theta}:=\left\{\theta\in\left[0,1\right]:\text{rot}(\theta)\in\right.$ $\left.\mathbf{R}\setminus\mathbf{Q}\right\}$. The objective is to prove the following result, suggested by Figure \ref{fig:Graphe_transl_num}.

\begin{prop}
\label{prop:Dim_Hausdorff}
    The Hausdorff dimension of $\mathcal{K}_{\Theta}$ is 0.
\end{prop}

An important remark is that we can parametrize the family $\{T_{\theta}\}$ by the singularity $s(\theta)$ instead of $\theta$, without affecting the Hausdorff dimension of $\mathcal{K}_{\Theta}$, since $s$ is a Lipschitz continuous function of $\theta$. In order to prove Proposition \ref{prop:Dim_Hausdorff}, we shall consider a larger class of g-AIETs. This class, together with the renormalization process we are going to introduce, is a minor adaptation and generalization of the ones presented in \cite{BoulangerFougeronGhazouani}. However, the core of the results and proofs remains unchanged. For instance, using \cite{BoulangerFougeronGhazouani}, one could prove straightforwardly that the Lebesgue measure of $\mathcal{K}_{\Theta}$ is 0. The true novelty is to be found in Proposition \ref{prop:Dim_Hausdorff_modele}.

\begin{cor}
    For almost all $\left(x,\theta\right)$, the geodesic starting at $\left(x,\theta\right)$ is attracted by a closed geodesic.
\end{cor}

Let $0<\lambda,\mu\leq1/2$, and let $\mathcal{I}(\lambda,\mu)$ be the family of g-AIETs $f$ defined on some $I$, with two intervals of definition, each one mapped to the opposite end of $I$, and with scaling factors $\lambda$ and $\mu$ (refer to Figure \ref{fig:Restrict_deux_interv}).

\begin{figure}[ht]
\centering
\includegraphics[scale=0.4]{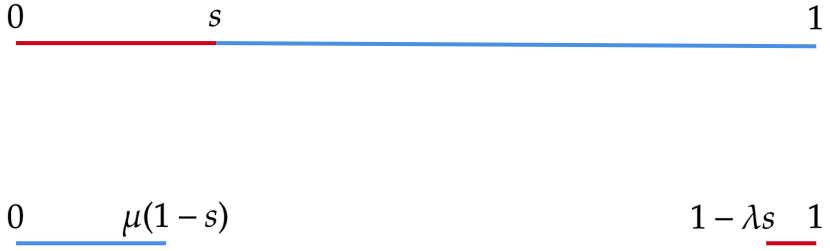}
\caption{An element of $\mathcal{I}(\lambda,\mu)$ defined on $\left[0,1\right]$}
\label{fig:Restrict_deux_interv}
\end{figure}

We shall construct a renormalization operator $\mathcal{R}$ on $\{\mathcal{I}(\lambda,\mu)\}_{0<\lambda,\mu\leq1/2}$ which is an adaptation of the well-known Rauzy-Veech induction. So let $f\in\mathcal{I}(\lambda,\mu)$ be defined on some interval $I$, and call $A$ and $B$ the left and right intervals of definition respectively. We denote by $l_{A}$ and $l_{B}$ the lengths of $A$ and $B$ respectively. We distinguish three cases:
\begin{enumerate}
    \item $B\subsetneq f(A)$, or, equivalently, $l_{B}<\lambda l_{A}$. In this situation, we consider the first return map on $A$, which belongs to $\mathcal{I}\left(\lambda,\lambda\mu\right)$. This process is called a right Rauzy-Veech induction. The length vector $(l_{A},l_{B})$ is transformed into
    \[
    \begin{pmatrix}
        l_{A}'  \\
        l_{B}' 
    \end{pmatrix}
    =R_{\lambda,\mu}
    \begin{pmatrix}
        l_{A}  \\
        l_{B}
    \end{pmatrix},
    \]
    where $R_{\lambda,\mu}:=\big(
    \begin{smallmatrix}
        1 & -1/\lambda \\
        0 & 1/\lambda
    \end{smallmatrix}
    \big)$. Note that $B\subsetneq f(A)$ is equivalent to both components of $R_{\lambda,\mu}\bigl(
    \begin{smallmatrix}
        l_{A} \\
        l_{B}
    \end{smallmatrix}
    \bigr)$ being positive.
    \item $A\subsetneq f(B)$, or, equivalently, $l_{A}<\mu l_{B}$. In this situation, we consider the first return map on $B$, which belongs to $\mathcal{I}(\lambda\mu,\mu)$. This process is called a left Rauzy-Veech induction. The length vector $(l_{A},l_{B})$ is transformed into
    \[
    \begin{pmatrix}
        l_{A}'  \\
        l_{B}' 
    \end{pmatrix}
    =L_{\lambda,\mu}
    \begin{pmatrix}
        l_{A}  \\
        l_{B}
    \end{pmatrix},
    \]
    where $L_{\lambda,\mu}:=\big(
    \begin{smallmatrix}
        1/\mu & 0 \\
        -1\mu & 1
    \end{smallmatrix}
    \big)$. Note that $A\subsetneq f(B)$ is equivalent to both components of $L_{\lambda,\mu}\bigl(
    \begin{smallmatrix}
        l_{A} \\
        l_{B}
    \end{smallmatrix}
    \bigr)$ being positive.
    \item In other cases, i.e. when $\lambda l_{A}\leq l_{B}\leq l_{A}/\mu$, $f$ has a periodic point of period $2$ and the renormalization is not defined.
\end{enumerate}

This renormalization operator $\mathcal{R}$ can be iterated until a periodic point of $f$ is found. We denote by $\mathcal{R}^{\circ n}f$ the $n$-th renormalization of $f$, if defined. Note that $f$ has a periodic orbit (which may contain the singularity) if and only if the induction eventually stops. The scaling factors of $\mathcal{R}^{\circ n}f$ are denoted $\lambda_{n}$ and $\mu_{n}$. We have $\lambda_{0}=\lambda$, $\mu_{0}=\mu$, and for $i=0,\dots,n-1$,

\[
\lambda_{i+1}=\left\{
\begin{array}{ccc}
    \lambda_{i}\mu_{i} & \text{if} & w_{i}=L  \\
    \lambda_{i} & \text{if} & w_{i}=R
\end{array}
\right.
\quad\text{and}\quad\mu_{i+1}=\left\{
\begin{array}{ccc}
    \mu_{i} & \text{if} & w_{i}=L  \\
    \lambda_{i}\mu_{i} & \text{if} & w_{i}=R
\end{array}.
\right.
\]

In addition, the length vector of $\mathcal{R}^{\circ n}f$ is obtained by multiplying the vector $(l_{A},l_{B})$ by the matrix $M_{n}$, which is defined by the induction $M_{0}=\text{Id}$ and, for $i=0,\dots,n-1$,

\[
M_{i+1}=\left\{
\begin{array}{ccc}
    L_{\lambda_{i},\mu_{i}}M_{i} & \text{if} & w_{i}=L  \\
    R_{\lambda_{i},\mu_{i}}M_{i} & \text{if} & w_{i}=R
\end{array}.
\right.
\]

In the following, we fix $\lambda$ and $\mu$, and since the interval of definition $I$ will not matter for the discussion, we fix $I=[0,1]$. We then consider the family $\{f_{s}\}_{s\in[0,1]}$, where $f_{s}$ is the element in $\mathcal{I}(\lambda,\mu)$ with singularity at $s$. To distinguish the dynamical and the parameter spaces, we denote the latter by $\mathcal{I}=[0,1]$. Denoting $\mathcal{K}$ the set of parameters $s$ such that the induction applied to $f_{s}$ never stops, we will prove the following result.

\begin{prop}
\label{prop:Dim_Hausdorff_modele}
    The Hausdorff dimension of $\mathcal{K}$ is 0.
\end{prop}

For a finite word $w$ in the alphabet $\{L,R\}$, we denote by $\mathcal{I}(w)$ the set of parameters $s$ such that the Rauzy-Veech induction defined above starts with $w$ (where $L$ and $R$ stand respectively for ``Left" and ``Right"). If $\emptyset$ denotes the empty word, $\mathcal{I}(\emptyset)$ is simply $\mathcal{I}$. Similarly, we define $\mathcal{H}(w)$ as the set of parameters $s$ such that the induction follows $w$ and stops.

\begin{lemma}
\label{lemma:Taille_trou_renorm}
    Let $w$ be word of length $n\in\mathbf{N}$. Then, if we normalize $\mathcal{I}(w)$ to $[0,1]$,
    \[
     \mathcal{H}(w)=\biggl[\frac{1}{1+\eta(w)/\mu_{n}},\frac{1}{1+\eta(w)\lambda_{n}}\biggr],
    \]
    where $\eta(w)$ is a constant depending on $w$ satisfying $1\leq\eta(w)\leq2$.
\end{lemma}

The proof of Lemma \ref{lemma:Taille_trou_renorm} goes exactly like in \cite{BoulangerFougeronGhazouani}. As a direct consequence, we get, for a word $w$ of length $n\in\mathbf{N}$,
\begin{equation}
\label{eq:proporI/H}
\frac{\lvert\mathcal{H}(w)\rvert}{\lvert\mathcal{I}(w)\rvert}\geq\frac{1}{1+2\lambda_{n}}-\frac{1}{1+1/\mu_{n}}.
\end{equation}
Moreover, $\mathcal{I}(w)\setminus \mathcal{H}(w)$ consists of two intervals, and if $\mathcal{J}$ is one of them, $\lvert\mathcal{J}\rvert/\lvert\mathcal{I}(w)\rvert\leq1/2$.

\begin{rmk}
    If $(w_{n})_{n\in\mathbf{N}}$ is a sequence of words of lengths $n$, then the numbers of $L$ and $R$ in $w_{n}$ both tend to infinity as $n\rightarrow+\infty$ if and only if $\lvert\mathcal{H}(w_{n})\rvert/\lvert\mathcal{I}(w_{n})\rvert\rightarrow1$.
\end{rmk}

We also define, for any $k\in\mathbf{N}$, the set $\mathcal{H}_{k}:=\bigcup_{\lvert w\rvert\leq k}\mathcal{H}(w)$, and also $\mathcal{H}:=\bigcup_{k}\mathcal{H}_{k}$. By definition, $\mathcal{K}=\mathcal{I}\setminus\mathcal{H}$. We are now able to give the proofs of Propositions \ref{prop:Dim_Hausdorff_modele} and \ref{prop:Dim_Hausdorff}.

\begin{proof}[Proof of Proposition \ref{prop:Dim_Hausdorff_modele}]
    A natural cover of $\mathcal{K}$ to consider is $\mathcal{K}_{n}:=\mathcal{I}\setminus\mathcal{H}_{n}$ for $n\in\mathbf{N}$. It is clear that $\mathcal{H}_{n+1}\setminus\mathcal{H}_{n}$ consists of $2^{n+1}$ disjoint intervals, and since $\mathcal{K}_{0}$ has two components, $\mathcal{K}_{n}$ is made of $2^{n+1}$ disjoint intervals. Let $w$ be a word of length $n$ and $\mathcal{J}$ one of the intervals intervals of $\mathcal{K}_{n}$. If $\delta_{j}^{w}$ denotes the proportion removed at the $j$-th step of the renormalization induced by $w$, the length of $\mathcal{J}$ is at most
    \[
    \prod_{j=0}^{n}(1-\delta_{j}^{w})\leq\frac{1}{2^{n+1}}.
    \]
    Now, let $l\in\mathbf{N}^{*}$, and set $\delta_{l}:=1-e^{-l^{2}}$. We consider the cover $\mathcal{K}_{lN}$ of $\mathcal{K}$ for some $N\in\mathbf{N}^{*}$, consisting of $2^{lN+1}$ intervals of length at most $1/2^{lN+1}$. There exists $n\in\mathbf{N}$ such that a finite word $w$ satisfies $\lvert H\left(w\right)\rvert/\lvert I\left(w\right)\rvert<\delta_{l}$ if and only if either the number of $L$ or the number of $R$ in $w$ is at most $n$. Therefore, if $lN\geq2n+1$, the number of words of length $lN+1$ satisfying $\lvert H\left(w\right)\rvert/\lvert I\left(w\right)\rvert<\delta_{l}$ is at most
    \[
    \sum_{j=0}^{n}\left(
    \begin{array}{c}
        lN+1 \\
        j 
    \end{array}\right)
    +\sum_{j=lN+1-n}^{lN+1}\left(
    \begin{array}{c}
        lN+1 \\
        j 
    \end{array}\right)
    =2\sum_{j=0}^{n}\left(
    \begin{array}{c}
        lN+1 \\
        j 
    \end{array}\right)=\frac{2}{n!}(lN)^{n}(1+\littleo{N\rightarrow+\infty}(1)).
    \]
    This implies that, among the $2^{(l+1)N+1}$ intervals composing $K_{(l+1)N}$, at least $2^{N}(2^{lN+1}-O(N^{n}))=2^{(l+1)N+1}(1-o(1))$ of them are smaller than $(1-\delta_{l})^{N}/2^{lN+1}$. This way, we obtain that, for any $\varepsilon>0$, the Hausdorff outer measure of dimension $\varepsilon$ bounded by $1/2^{(l+1)N+1}$ of $\mathcal{K}$ is less than
    \begin{equation}
    \label{Eq:Mes_Hausdorff}
      \frac{2^{1-\varepsilon}(lN)^{n}}{n!}(1+o(1))\left[2^{1-(l+1)\varepsilon}\right]^{N}+2^{1-\varepsilon}(1-o(1))\left[2^{l+1-l\varepsilon}(1-\delta_{l})^{\varepsilon}\right]^{N}.  
    \end{equation}
    For $\varepsilon$ larger than
    \[
    M_{l}:=\max\left(\frac{1}{l+1},\frac{(l+1)\log2}{l\log2-\log(1-\delta_{l})}\right),
    \]
    the quantity (\ref{Eq:Mes_Hausdorff}) goes to zero when $N$ tends to infinity, proving that $\text{dim}_{H}\mathcal{K}\leq M_{l}$ for all $l\in\mathbf{N}$. Finally, since $M_{l}\rightarrow0$ as $l\rightarrow+\infty$, we get $\text{dim}_{H}\mathcal{K}=0$.
\end{proof}

\begin{proof}[Proof of Proposition \ref{prop:Dim_Hausdorff}]
We restrict the parameter $\theta$ to the interval 
\[
\Theta:=[\arctan(13/21),\arctan(16/17)],
\]
otherwise $T_{\theta}$ maps the whole interval $\left[0,1\right]$ into the interior of one of the two intervals of definition, and the rotation number is 0. For any $\theta\in\Theta$, $T_{\theta}$ maps $\left[0,1\right]$ into 
\[
\left[\lim_{x\rightarrow s(\theta)^{+}}T_{\theta}\left(x\right),\lim_{x\rightarrow s(\theta)^{-}}T_{\theta}\left(x\right)\right].
\]
By restricting $T_{\theta}$ to this sub-interval, we obtain a g-AIET with two intervals, each one contracted by a factor $1/16$ (as before) and mapped to the opposite end. As a consequence, $T_{\theta}\in\mathcal{I}\left(1/16,1/16\right)$, and the result follows from Proposition \ref{prop:Dim_Hausdorff_modele}.
\end{proof}

\section{Trajectories for the corresponding vector field}

\label{sec:Chp_vect}

In this section, we study the homogeneous vector field $\mathbf{v}$ in $\mathbf{C}^{2}$ associated with $\mathcal{S}$. Recall that in suitable coordinates it is given by $\mathbf{v}\left(x,y\right)=v_{1}\left(x,y\right)\partial_{x}+v_{2}\left(x,y\right)\partial_{y}$, with
\[
\left\{
\begin{array}{l}
\displaystyle v_{1}\left(x,y\right) = -\left(\frac{\pi}{2}+i\log2\right)x^{2}+\left(\frac{3\pi}{4}-\frac{i}{2}\log2\right)xy\\[3pt]
\displaystyle v_{2}\left(x,y\right) = \left(\frac{3\pi}{2}-i\log2\right)xy-\left(\frac{5\pi}{4}+\frac{i}{2}\log2\right)y^{2}
\end{array}
\right.
.
\]

The three characteristic directions are $\left\{x=0\right\}$, $\left\{y=0\right\}$ and $\left\{x=y\right\}$. Recall also that $g:\mathbf{C}^{2}\setminus\{(0,0)\}\rightarrow\widehat{\mathbf{C}}$ denotes the projection defined by $g\left(x,y\right)=x/y=:z$. Moreover, Theorem \ref{thm:AbateTovena} gives a meromorphic affine structure on $\widehat{\mathbf{C}}$, denoted $\mathcal{S}_{\mathbf{v}}$, such that $g$ maps trajectories of $\mathbf{v}$ to geodesics of $\mathcal{S}_{\mathbf{v}}$. We justified in section \ref{sec:Intro} that $\mathcal{S}_{\mathbf{v}}$ and $\mathcal{S}$ are equivalent as meromorphic affine structures. With this identification, the singularities at $A$ (or $C$), $B$ and $D$ correspond respectively to $1$, $0$ and $\infty$. To deal with the singularity at infinity, it may be necessary to consider the projection $1/g$ rather than $g$, or, in other words, to look at $\widehat{\mathbf{C}}$ in the coordinate $w=1/z$.

We also defined the polynomials $P(x,y)=x\,v_{2}(x,y)-y\,v_{1}(x,y)$, and $p(z)=P(x,y)/y^{3}=2\pi z(z-1)$. Let $\gamma=(\gamma_{1},\gamma_{2})$ be a trajectory for $\mathbf{v}$, and $\delta:=g\circ\gamma=\gamma_{1}/\gamma_{2}$ be its projection under $g$.  Following \cite{BuffRaissy}, for a vector field of degree $k$,
\begin{equation}
\label{eq:egal_gamma2}
\left(\gamma_{2}(t)\right)^{k-1}=f(t):=-\frac{\delta'(t)}{p(\delta(t))}.
\end{equation}
Note that $p$ has simple roots at 0 and 1. In consequence, if $s$ is 0 or 1, then $p(\delta(t))\sim p'(s)(\delta(t)-s)$ when $\delta(t)$ is close to $s$. However, if $s$ is $\infty$, then, working in the $w$ coordinate we get $p(1/\delta(t))\sim1/\pi\delta(t)^{2}$. We are now ready to prove Proposition \ref{prop:accumulation}, whose statement we recall.

\begin{repprop}{prop:accumulation}
    The accumulation set of any regular trajectory contains the origin.
\end{repprop}

\begin{proof}
    As previously established, the speed of regular geodesics tends to zero. However, it is clear that such a geodesic has to leave any (small enough) neighborhood of any singularity. This gives a sequence $(t_{k})_{k\in\mathbf{N}}$ such that $(f(t_{k}))_{k}$ tends to 0 as $k$ tends to infinity. Therefore, by (\ref{eq:egal_gamma2}), $\gamma_{2}(t_{k})\rightarrow0$, and since $\delta(t_{k})=\gamma_{1}(t_{k})/\gamma_{2}(t_{k})$ is far from $\infty$, $(\gamma_{1}(t_{k}))_{k}$ must also tend to 0 as $k$ tends to infinity.
\end{proof}

\begin{lemma}
    If a geodesic in $\mathcal{S}$ accumulates $B$ or $D$, then it must converge to it. As a consequence, geodesics in $\widehat{\mathbf{C}}$ either are bounded or escape to infinity.
\end{lemma}

\begin{proof}
    Let $\delta$ be a geodesic in $\mathcal{S}$ that does not hit a singularity. It suffices to prove the Lemma for $\delta$ regular, as the backward-time geodesic of an irregular geodesic is regular. From what precedes, after a while it only hits the edges $\left]B,C\right[$ and $\left]C,D\right[$. Assume, without loss of generality, that $\delta$ emanates from $x_{0}\in[A,B]$ with angle $\theta_{0}\in[0,\pi/4]$. The only way for $\delta$ to get close to $D$ is that $x_{0}$ is close to $B$, $\theta_{0}$ close to $\pi/4$, and $\delta$ hits $[B,C]$. Hence it is enough to prove that $\delta$ cannot accumulate $B$. If $\theta_{0}\leq\widetilde{\theta}$ or if $x_{0}\geq2\tan(\theta)-1$, then, by the second formula in (\ref{eq:formules_T_theta}), $T_{\theta_{0}}(x_{0})=x_{3}=(x_{0}-1)/16+1-(3\tan\theta_{0}+1)/4$. Because this function is increasing in $x_{0}$ and decreasing in $\theta_{0}$, its maximum is attained at $x_{0}=1$ and $\theta_{0}=0$. This gives $x_{3}\leq3/4$. On the other hand, if $\theta_{0}\geq\widetilde{\theta}$ and $x_{0}\leq2\tan(\theta_{0})-1$, then, by the first formula in (\ref{eq:formules_T_theta}) and the same argument as before, $x_{3}\leq15/16$. In any case, $x_{3}$ is far from $B$. Now, it is clear that $x_{4}$ will be farther from $B$. Also, $\theta_{2}$ is non-positive, and hence $x_{2}$ is farther from $B$ as well. All this together implies that $\delta$ cannot accumulate $B$.
\end{proof}

Now, our objective is to prove Theorem \ref{thm:main}. To achieve this, we examine the connection between the rotation number of a geodesic and the accumulation set of its associated trajectory of $\mathbf{v}$. Let $\theta\in[0,\pi/4]$, $\delta$ be a regular geodesic in $\mathcal{S}$, $\left\{x_{n}\right\}_{n\in\mathbf{N}}$ the induced infinite forward $T_{\theta}$-orbit, and $\gamma=(\gamma_{1},\gamma_{2})$ the corresponding trajectory. Define $\left(t_{n}\right)_{n\in\mathbf{N}}$ by $t_{0}=0$ and $\delta(t_{n})=x_{n}$, and let $v_{n}=\delta'(t_{n})=\lambda^{n}\delta'(t_{0})$ be the speed vector of $\delta$ at time $t_{n}$. Our goal is to describe the omega-limit set of the sequence $(f(t_{n}))_{n}$. By Lemma \ref{RestrictABAD}, $\delta'(t_{n})\rightarrow0$ as $n\rightarrow+\infty$. The crucial tool is the following general lemma about geodesics near a conical singularity, i.e. a singularity where the real part of the residue is less than 1.

\begin{lemma}
    \label{lem:Lemme_gen_sing_con}
    Let $\mathbf{v}$ be a homogeneous and non-dicritical vector field of degree $k\geq2$, and $\mathcal{S}$ be the induced meromorphic affine structure on $\widehat{\mathbf{C}}$. Assume 0 is a conical singularity of $\mathcal{S}$ with residue $1-\mu$, where $0<\operatorname{Re}(\mu)<1$. Let $\varepsilon>0$ be small enough and so that 0 is the only singularity of $\mathcal{S}$ inside $\mathbf{D}(0,\varepsilon)$. Assume also that we have two sequences of complex numbers $\{\zeta_{n}\}$ and $\{u_{n}\}$ tending to zero, and such that $u_{n}/\zeta_{n}\rightarrow l\in\widehat{\mathbf{C}}$. For each $n$, let $\delta_{n}:I_{n}\rightarrow\mathbf{D}(0,\varepsilon)$ be the maximal geodesic satisfying $\delta_{n}(0)=\zeta_{n}$ and $\delta_{n}'(0)=u_{n}$. Let also $\gamma_{n}:=(\gamma_{n}^{1},\gamma_{n}^{2})$ be the corresponding pieces of trajectories (meaning $\delta_{n}=\pi\circ\gamma_{n}$), and $\{t_{n}\}$ a sequence ``converging" to $t\in\mathbf{R}\cup\{\infty\}$ with $t_{n}\in I_{n}$ for all $n$. Then,
    \begin{enumerate}
        \item $\gamma_{n}^{1}(t_{n})\rightarrow0$;
        \item if $t\neq\infty$, $\left(\gamma_{n}^{2}(t_{n})\right)^{k-1}\rightarrow-l/p'(0)(\mu lt+1)$, with the convention $1/\infty=0$;
        \item if $t=\infty$ and $t_{n}+\zeta_{n}/\mu u_{n}\rightarrow x$, $x$ is necessarily in $\mathbf{R}\cup\{\infty\}$ and $\left(\gamma_{n}^{2}(t_{n})\right)^{k-1}\rightarrow-1/p'(0)\mu x$.
    \end{enumerate}
\end{lemma}

\begin{proof}
Let $\varphi$ be an affine chart from a neighborhood of $0$ in $\mathcal{S}$ to a neighborhood of $0$ in the sector $V:=\left\{z\in\mathbf{C}:\arg z\in\left[0,2\pi\operatorname{Re}(\mu)\right)\right\}$. Let also, for $j=\lfloor-1/2\pi\operatorname{Re}(\mu)\rfloor,\ldots,\lceil1/2\pi\operatorname{Re}(\mu)\rceil$, $\varphi_{j}:=e^{2ij\pi\mu}\varphi$, so that the total angle covered by the $\varphi_{j}$'s is at least $2\pi$. Let $\eta_{n}$ be the piece of geodesic defined in $V$ by $\eta_{n}=\varphi\circ\delta_{n}$, and extended by the $\varphi_{j}$'s if necessary. Then, for $t\in I_{n}$, $\eta_{n}(t)=v_{n}t+z_{n}$, where $z_{n}=\varphi(\zeta_{n})$ and $v_{n}=\varphi'(\zeta_{n})u_{n}$. Since, for any $j$,
\[
\varphi_{j}(\zeta)=e^{2ij\pi\mu}\varphi(\zeta)=e^{2ij\pi\mu}\zeta^{\mu}(1+O(\zeta))
\]
and
\[
\varphi_{j}'(\zeta)=e^{2ij\pi\mu}\varphi'(\zeta)=e^{2ij\pi\mu}\mu\zeta^{\mu-1}(1+O(\zeta)),
\]
we have that $v_{n}/z_{n}=\mu u_{n}/\zeta_{n}$. Note that the assumption $\operatorname{Re}(\mu)>0$ is essential, since it ensures that $z\mapsto z^{\mu}$ maps 0 to 0. Moreover,

\begin{equation}
    \label{Eq:Vitesse_sur_dist_carte_aff}
    \frac{\eta_{n}'(t_{n})}{\eta_{n}(t_{n})}=\frac{v_{n}}{v_{n}t_{n}+z_{n}}=\mu\frac{\delta_{n}'(t_{n})}{\delta_{n}(t_{n})}(1+O(\delta_{n}(t_{n}))).
\end{equation}

If $\delta_{n}(t_{n})\nrightarrow0$, then necessarily $t=\infty$. Indeed, the distance traveled by $\eta_{n}$ from time 0 to $T>0$ is $T\lvert v_{n}\rvert\rightarrow0$, which implies that $\eta_{n}(T)\rightarrow0$, and hence $\delta_{n}(T)\rightarrow0$, for any fixed $T>0$. In this case, Equation \ref{Eq:Vitesse_sur_dist_carte_aff} implies that $\delta_{n}'(t_{n})\rightarrow0$, whence $\gamma_{n}^{1}(t_{n}),\left(\gamma_{n}^{2}(t_{n})\right)^{k-1}\rightarrow0$. Let us check that this fits into case 3 with $x=\infty$. As mentioned above, the distance traveled by $\eta_{n}$ from time 0 to $t_{n}>0$ is $t_{n}\lvert v_{n}\rvert$, which implies that $t_{n}$ is equivalent, up to a multiplicative constant, to $1/\lvert v_{n}\rvert$ when $n$ tends to infinity. However, since $\zeta_{n}$ and $z_{n}$ go to 0 by assumption, $\lvert\zeta_{n}/\mu u_{n}\rvert=\lvert z_{n}/v_{n}\rvert=o(1/\lvert v_{n}\rvert)$, and therefore $t_{n}+\zeta_{n}/\mu u_{n}$ must go to infinity.

Now, if $\delta_{n}\left(t_{n}\right)\rightarrow0$, we automatically have $\gamma_{n}^{1}(t_{n})\rightarrow0$. Assume first that $t=\infty$ and $t_{n}+\zeta_{n}/\mu u_{n}\rightarrow x\in\widehat{\mathbf{C}}$. If $l\neq0$, then necessarily $x=\infty$, and the conclusion follows. If $l=0$, the only way for $x$ to be finite is that $\lvert v_{n}t_{n}+z_{n}\rvert=O(\lvert v_{n}\rvert)$. If $v_{n}$ makes an angle $\theta_{n}$ with the radial direction from $z_{n}$, we have that $\lvert v_{n}t_{n}+z_{n}\rvert\geq\lvert z_{n}\sin\theta_{n}\rvert$. But we have assumed that $\lvert v_{n}\rvert=o(\lvert z_{n}\rvert)$, so $\theta_{n}$ has to go to $0\text{ mod }\pi$. From this we conclude easily that $x\in\mathbf{R}\cup\{\infty\}$. As a consequence of Equation \ref{Eq:Vitesse_sur_dist_carte_aff}, 
\[
\frac{\delta_{n}'\left(t_{n}\right)}{\delta_{n}\left(t_{n}\right)}\isEquivTo{n\rightarrow+\infty}\frac{1}{\displaystyle\mu\left(t_{n}+\frac{z_{n}}{v_{n}}\right)}\longrightarrow\frac{1}{\displaystyle\mu x},
\]
and by Equation \ref{eq:egal_gamma2}, $\left(\gamma_{n}^{2}(t_{n})\right)^{k-1}\rightarrow-1/p'(0)\mu x$.

Finally, if $t\neq\infty$, we apply Equations \ref{Eq:Vitesse_sur_dist_carte_aff} and \ref{eq:egal_gamma2} to get
\[
\frac{\delta_{n}'\left(t_{n}\right)}{\delta_{n}\left(t_{n}\right)}\isEquivTo{n\rightarrow+\infty}\frac{1}{\displaystyle\mu\left(t_{n}+\frac{z_{n}}{v_{n}}\right)}\longrightarrow\frac{1}{\displaystyle\mu\left(t+\frac{1}{\mu l}\right)}=\frac{l}{\mu tl+1}
\]
and
\[
\left(\gamma_{n}^{2}(t_{n})\right)^{k-1}\longrightarrow-\frac{l}{p'(0)(\mu lt+1)}
\]
as desired.
\end{proof}

\begin{cor}
\label{cor:cercle_limite}
    Let $\delta$ be a regular geodesic in $\mathcal{S}$, $\left\{x_{n}\right\}_{n\in\mathbf{N}}$ the induced $T_{\theta}$-orbit, and $\gamma$ the corresponding trajectory. Define $\left(t_{n}\right)_{n\in\mathbf{N}}$ by $t_{0}=0$ and $\delta(t_{n})=x_{n}$, and let $v_{n}=\delta'(t_{n})=\lambda^{n}\delta'(t_{0})$ be the speed vector of $\delta$ at time $t_{n}$.\\
    If the sequence $(v_{n}/x_{n})_{n\in\mathbf{N}}$ admits a subsequential limit $l\in\widehat{\mathbf{C}}$, then necessarily $l\in(\widehat{\mathbf{C}}\setminus\mathbf{R})\cup\{0\}$. Moreover, if $l$ is finite, the omega-limit set of $\gamma$ contains the Euclidean circle $\{l/2\pi\mu_{0}(lt+1),t\in\mathbf{R}\cup\{\infty\}\}$ (inside $L_{0}$). If $l=\infty$, the omega-limit set of $\gamma$ contains the line directed by $1/\mu_{0}$ (inside $L_{0}$).
\end{cor}

\begin{proof}
Note first that, since $\theta\in[0,\pi/4]$, the sequence $(v_{n}/x_{n})_{n\in\mathbf{N}}$ cannot accumulate a non-zero real number. This implies in particular that case 3 of Lemma \ref{lem:Lemme_gen_sing_con} can only happen with $x=\infty$. Then, using Lemma \ref{lem:Lemme_gen_sing_con} and its proof, the sequences $(x_{n})_{n\in\mathbf{N}}$ and $(v_{n})_{n\in\mathbf{N}}$ in $\mathcal{S}$ correspond to sequences $(\zeta_{n})_{n\in\mathbf{N}}$ and $(u_{n})_{n\in\mathbf{N}}$ in $\mathcal{S}_{\mathbf{v}}$ respectively. The role of the affine chart $\varphi$ is played by the isomorphism between $\mathcal{S}$ and $\mathcal{S}_{\mathbf{v}}$. Moreover, $u_{n}/\zeta_{n}\rightarrow l/\mu_{0}$ when $n$ tends to infinity. We can then apply Lemma \ref{lem:Lemme_gen_sing_con} to $(\zeta_{n})_{n\in\mathbf{N}}$ and $(u_{n})_{n\in\mathbf{N}}$, and use the fact that $p'(0)=-2\pi$.
\end{proof}

This Corollary is a powerful tool to show the existence of circles in the omega-limit set of $\gamma$. In the following, $\lambda=1/16$ stands for the scaling ratio, and recall that $G=]a_{0},b_{0}[:=]\lambda\left(1-s\right),1-\lambda s[$ is the ``gap" at the image of $T_{\theta}$. It will be useful to see $T_{\theta}$ as a circle map by identifying $0$ and $1$, and to interpret the image of $0\,(\text{mod }1)$ as the gap $G:=]a_{0},b_{0}[:=]\lambda\left(1-s\right),1-\lambda s[$. Now, let $\theta\in\bigl]\widetilde{\theta},\pi/4\bigr[$, and $x_{0}\in[0,1]$ have a well defined orbit $(x_{n})_{n\in\mathbf{N}}$ under $T_{\theta}$. We define $\Lambda_{x_{0}}$ as the set of subsequential limits of $\left(\lambda^{n}/(x_{n}-s)\right)_{n}$. Note that $\Lambda_{x_{0}}$ always contains 0. The following two Propositions give a description of $\Lambda_{x_{0}}$ depending on the rotation number of $\theta$.

\begin{prop}
\label{prop:Val_adh_rot_rat}
The following holds for $\Lambda_{x_{0}}$, assuming $\text{rot}(\theta)\in\mathbf{Q}\,(\text{mod }1)$:
    \begin{enumerate}
        \item if $\theta\in\text{int}\left(\text{rot}^{-1}\mathbf{Q}\right)$, then $\Lambda_{x_{0}}=\left\{0\right\}$;
        \item if $\text{rot}(\theta)=p/q$ and $\theta\in\partial\,\text{rot}^{-1}(p/q)$, then $\Lambda_{x_{0}}=\left\{0,l\right\}$ for some $l\in\mathbf{R}^{*}$.
    \end{enumerate}
\end{prop}

\begin{proof}
\textbf{Case 1.} In this situation, $\left\{x_{n}\right\}_{n\in\mathbf{N}}$ is attracted by a periodic cycle, so it stays away from the singularities. Hence $f\left(t_{n}\right)$ goes to $0$ as $n$ tends to $+\infty$. Thus, $\gamma_{2}(t)\rightarrow0$, but, since $\delta(t)=\gamma_{1}(t)/\gamma_{2}(t)\nrightarrow0$, we also have $\gamma_{1}(t)\rightarrow0$. Note that the condition $\theta\in\text{int}\left(\text{rot}^{-1}\mathbf{Q}\right)$ is indeed generic with respect to the Lebesgue measure.

\textbf{Case 2.} Here we have a saddle connection: $s=s(\theta)$ is periodic of period $q$ and for one of the two possible extensions of $T_{\theta}$ through $s$. This situation is similar to the previous one, and near the other points of the cycle $f$ tends to 0. However, in a neighborhood of $s$, the sequence $f\left(t_{n}\right)$ can be understood with the renormalization process discussed above. After less than $q$ steps, we are reduced to the schematic picture described in Figure \ref{fig:Renorm_sing_per}. Therefore, it is easy to see that for some $k\in\mathbf{N}$, the subsequence $\left(v_{k+qn}/x_{k+qn}\right)_{n\in\mathbf{N}}$ is constant, and that $l=v_{k}/x_{k}$ is the only non-zero subsequential limit of $\left(v_{n}/x_{n}\right)_{n\in\mathbf{N}}$.

\end{proof}

\begin{figure}[ht]
\centering
\includegraphics[scale=0.3]{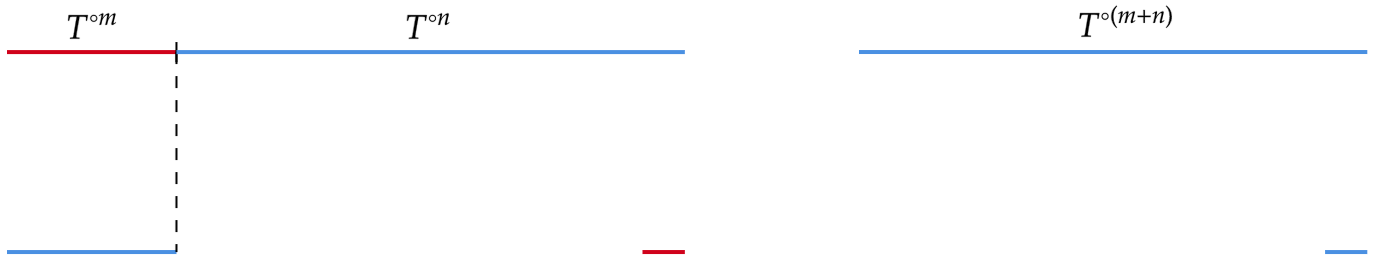}
\caption{Renormalization when the singularity is ``periodic", with period $q=m+n$}
\label{fig:Renorm_sing_per}
\end{figure}

\begin{prop}
\label{prop:Val_adh_rot_irrat}
    Let $\theta\in[0,\pi/4]$ with $\text{rot}(\theta)\in\mathbf{R}\setminus\mathbf{Q}$ and $x_{0}\in[0,1]$. The following holds for $\Lambda_{x_{0}}$, depending on $x_{0}$:
    \begin{enumerate}
        \item if $x_{0}\in K_{\theta}$ is not a boundary point of a gap, then $\Lambda_{x_{0}}=\left\{0,\infty\right\}$;
        \item if $x_{0}=a_{k}$ (respectively $x_{0}=b_{k}$), then $\Lambda_{x_{0}}=\left\{0,\infty,-1/\lambda^{k}(1-\lambda)\right\}$ (respectively $\Lambda_{x_{0}}=\left\{0,\infty,\right.$ $\left.1/\lambda^{k}(1-\lambda)\right\}$;
        \item if $x_{0}\in T_{\theta}^{\circ k}(G)$, then $\Lambda_{x_{0}}=\left\{0,1/(x_{0}-a_{k}),1/(x_{0}-b_{k})\right\}$.
    \end{enumerate}
\end{prop}

Since the rotation number of $T_{\theta}$ is irrational, $T_{\theta}$ is semi-conjugated to the rotation by $\theta$, and the closure of $G$ corresponds to the second iterate of $s$. Since $\theta$ is irrational, $s\notin\bigcup_{n\geq0}\,T_{\theta}^{\circ n}\left(\overline{G}\right)$. Note that $G=\text{int}\left(\left[0,1\right]\setminus T_{\theta}\left(\left[0,1\right]\right)\right)$, and therefore $K_{\theta}=\left[0,1\right]\setminus\bigcup_{n\geq0}\,T_{\theta}^{\circ n}\left(G\right)$. For the proof of Proposition \ref{prop:Val_adh_rot_irrat}, we will need a series of Lemmas.

\begin{lemma}
\label{lem:AdhK}
     $K_{\theta}=\overline{\left\{T_{\theta}^{-n}\left(s\right),n\in\mathbf{N}\right\}}$.
\end{lemma}

\begin{proof}
    It suffices to prove that for any $x\in K_{\theta}$ and any $\varepsilon>0$, there exists $y\in I_{\varepsilon}:=\left]x-\varepsilon,x+\varepsilon\right[$ and $n\in\mathbf{N}$ so that $T_{\theta}^{\circ n}\left(y\right)=s$. So, let $x\in K_{\theta}$ and $\varepsilon>0$. Remember that $\bigcup_{n\geq0}\,T_{\theta}^{\circ n}\left(G\right)$ is dense in $\left[0,1\right]$, and by injectivity of $T_{\theta}$, $I_{\varepsilon}$ must contain $T_{\theta}^{\circ n}\left(G\right)$ for some $n\in\mathbf{N}$. Let $\eta>0$ be small enough. Again, there exist infinitely many $N>n$ such that $T_{\theta}^{\circ N}\left(G\right)\subset I_{\eta}$, hence satisfying
    \begin{equation}
    \label{eq:sizeGeta}
        \left\lvert T_{\theta}^{\circ\left(N\right)}\left(G\right)\right\rvert=\lambda^{N}\left(1-\lambda\right)<2\eta.
    \end{equation}
    Now assume that $I_{\varepsilon}$ does not contain any preimage of $s$. Then, its infinite orbit is well defined, and $T_{\theta}^{\circ\left(N-n\right)}\left(I_{\varepsilon}\right)\cap I_{\eta}\neq\emptyset$. In fact, $\lvert T_{\theta}^{\circ\left(N-n\right)}\left(I_{\varepsilon}\right)\rvert=\lambda^{N-n}\varepsilon<\left(\varepsilon-\eta\right)$ if $\eta$ is small enough and $N$ large enough. But this implies that $I_{\varepsilon}$ is mapped into itself by $T_{\theta}^{\circ\left(N-n\right)}$, which is not possible since $T_{\theta}$ has no periodic point.
\end{proof}

In the following, we will think of $T_{\theta}$ as acting on the circle $S^{1}\cong\mathbf{R}/\mathbf{Z}$ equipped with the counter-clockwise orientation. Therefore, an interval $[x,y]$ in $S^{1}$ is the arc traveled when moving counter-clockwise from $x$ to $y$.

\begin{definition}
    Let $x\in S^{1}$. A preimage $T_{\theta}^{-n}(s)$ of $s$ is called minimal for $x$ if $\{T_{\theta}^{-k}(s):k=0,$ $\dots,n-1\}$ is contained in $[x,T_{\theta}^{-n}(s)]$ or in $[T_{\theta}^{-n}(s),x]$.
\end{definition}

In other words, a preimage of $s$ is minimal if one of the intervals $[x,T_{\theta}^{-n}(s)]$ and $[T_{\theta}^{-n}(s),x]$ can be iterated $n$ times without hitting the singularity.

\begin{lemma}
\label{lem:exist_infty}
    If $x_{0}\in K_{\theta}$, then there exists a subsequence $(x_{n_{k}})_{k\in\mathbf{N}}$ such that
    \[
    \frac{\lambda^{n_{k}}}{\left\lvert x_{n_{k}}-s\right\rvert}\xrightarrow[k\rightarrow+\infty]{}+\infty.
    \]
\end{lemma}

\begin{proof}
    Lemma \ref{lem:AdhK} guarantees the existence of a sequence $(T_{\theta}^{-n_{k}}(s))_{k\in\mathbf{N}^{*}}$ of minimal preimages for $x_{0}$ such that $\lvert T_{\theta}^{-n_{k}}(s)-x_{0}\rvert<1/k$. By definition, each interval $[T_{\theta}^{-n_{k}}(s),x_{0}]$ (or $[x_{0},T_{\theta}^{-n_{k}}(s)]$) can be iterated $n_{k}$ times without encountering the singularity. As a consequence, $\left\lvert x_{n_{k}}-s\right\rvert<\lambda^{n_{k}}/k$, and the conclusion follows.
\end{proof}

\begin{lemma}
\label{lem:Preim_non_min}
    Let $x_{0}\in K_{\theta}$ and $(T_{\theta}^{-n_{k}}(s))_{k\in\mathbf{N}}$ be a sequence of non-minimal preimages of $s$ for $x_{0}$. Then,
    \[
    \frac{\lambda^{n_{k}}}{\left\lvert x_{n_{k}}-s\right\rvert}\xrightarrow[k\rightarrow+\infty]{}0.
    \]
\end{lemma}

\begin{proof}
    For each $k$, let $r_{k}$ (respectively $r_{k}'$) be the largest integer that is smaller than $n_{k}$ and such that $T_{\theta}^{-r_{k}}(s)\in[x_{0},T_{\theta}^{-n_{k}}(s)]$ (respectively $T_{\theta}^{-r_{k}'}(s)\in[T_{\theta}^{-n_{k}}(s),x_{0}]$). Since the second iterate of the singularity is $G$, it is not difficult to see that, assuming $n_{k}\geq\max(r_{k},r_{k}')+2$, the intervals $[x_{n_{k}},s]$ and $[s,x_{n_{k}}]$ contain respectively a gap of size $\lambda^{n_{k}-(r_{k}+2)}(1-\lambda)$ and $\lambda^{n_{k}-(r_{k}'+2)}(1-\lambda)$ (see Figure \ref{fig:Preim_non_minim}). If $n_{k}\geq\max(r_{k},r_{k}')+2$, one of the intervals $[x_{n_{k}},s]$ and $[s,x_{n_{k}}]$ contains $[s,T_{\theta}(s)]$ or $[T_{\theta}(s),s]$. Thus, either $\lambda^{n_{k}}/\lvert x_{n_{k}}-s\rvert\leq\lambda^{\min(r_{k},r_{k}')+2}/(1-\lambda)$ or $\lambda^{n_{k}}/\lvert x_{n_{k}}-s\rvert\leq\max(\lambda^{\min(r_{k},r_{k}')+2}/(1-\lambda),\lambda^{n_{k}}/\lvert s,T_{\theta}(s)\rvert)$. In any case, it only remains to prove that $\min(r_{k},r_{k}')$ goes to zero when $k$ goes to infinity. So let's assume, by contradiction, that $\min(r_{k},r_{k}')=r$ for all $k$ large enough. Assume also, without loss of generality, that $r_{k}=r$ if $k$ is sufficiently large. This implies that $[x_{0},T_{\theta}^{-n_{k}}(s)]$ contains $T_{\theta}^{-r}(s)$ for all $k$ large enough. But the interval $[x_{0},T_{\theta}^{-r}(s)]$ contains infinitely many preimages of $s$, leading to a contradiction, since $r$ is supposed to be the largest integer such that $T_{\theta}^{-r}(s)\in[x,T_{\theta}^{-n_{k}}(s)]$ for $k$ sufficiently large.
\end{proof}

\begin{figure}[ht]
\centering
\includegraphics[scale=0.18]{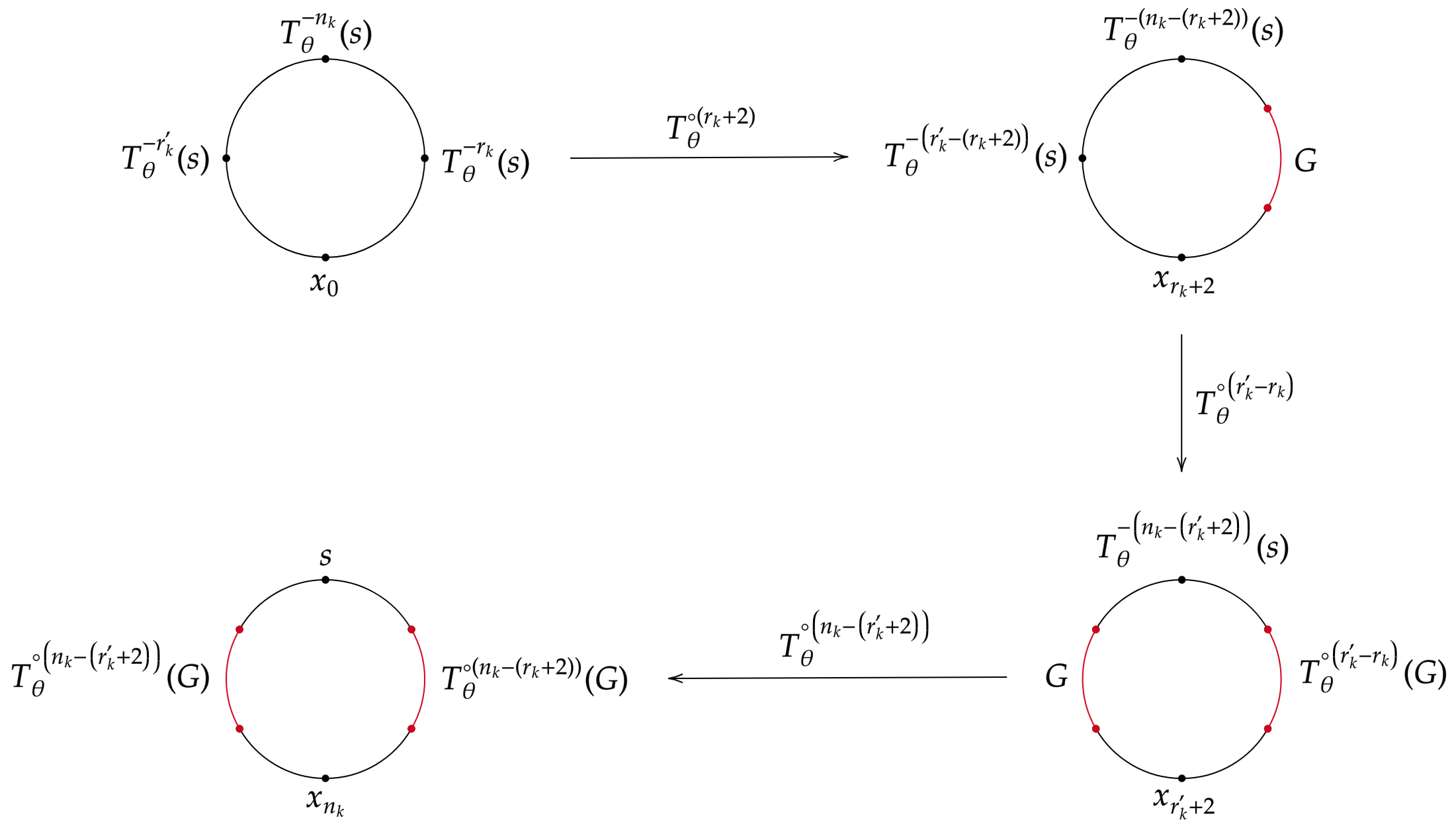}
\caption{\centering Schematic picture describing the $n_{k}$-th iterate of $T_{\theta}$ when $T_{\theta}^{-n_{k}}(s)$ is not minimal for $x_{0}$. We have assumed $r_{k}'>r_{k}+2$ and $n_{k}\geq r_{k}'+2$}
\label{fig:Preim_non_minim}
\end{figure}

\begin{proof}[Proof of Proposition \ref{prop:Val_adh_rot_irrat}] The first statement follows directly from Lemmas \ref{lem:exist_infty} and \ref{lem:Preim_non_min}. Now, let $x_{0}\in[a_{k},b_{k}]$ for some $k$, and let $(T_{\theta}^{-n_{j}}(s))_{j}$ be a sequence of preimages of $s$. Since any preimage of $s$ that is minimal for $a_{k}$ is also minimal for $b_{k}$, we can assume, up to extracting a subsequence, that all $T_{\theta}^{-n_{j}}(s)$ are either minimal or non-minimal for both $a_{k}$ and $b_{k}$. In the former case, $(T_{\theta}^{-n_{j}}(s))_{j}$ must converge to $a_{k}$ or to $b_{k}$. If the limit is $a_{k}$, then $x_{n_{j}}-s\sim_{j\rightarrow+\infty}x_{n_{j}}-a_{k+n_{j}}=\lambda^{n_{j}}(x_{0}-a_{k})$, and hence $1/(x_{0}-a_{k})\in\Lambda_{x_{0}}$. Similarly, if the limit is $b_{k}$, then $1/(x_{0}-b_{k})\in\Lambda_{x_{0}}$. Finally, if the $T_{\theta}^{-n_{j}}(s)$ are non-minimal, then $\lvert x_{n_{j}}-s\rvert\geq\min(\lvert a_{k+n_{j}}-s\rvert,\lvert b_{k+n_{j}}-s\rvert)=\lambda^{n_{j}}(x_{0}-a_{k})$. Therefore, by Lemma \ref{lem:Preim_non_min}, $\lambda^{n_{j}}/\lvert x_{n_{j}}-s\rvert\leq\lambda^{n_{j}}/\min(\lvert a_{k+n_{j}}-s\rvert,\lvert b_{k+n_{j}}-s\rvert)\longrightarrow_{j\rightarrow+\infty}0$.
\end{proof}

Now Theorem \ref{thm:main} follows easily. Indeed, to any point $\left(x,y\right)\in\mathbf{C}^{2}\setminus\mathcal{C}$ corresponds a unique trajectory $\gamma$ for $\mathbf{v}$, hence a unique geodesic $\delta$ in $\mathcal{S}$ and a well-defined rotation number $\text{rot}$. This relationship allows us to lift $\text{rot}$ to a continuous function $\text{Rot}$ defined on $\mathbf{C}^{2}\setminus\mathcal{C}$ and constant along the real-time trajectories of $\mathbf{v}$. The dynamics of the trajectories depending on their rotation number derives from Corollary \ref{cor:cercle_limite} and Propositions \ref{prop:Val_adh_rot_rat} and \ref{prop:Val_adh_rot_irrat}. 

Let $\mathcal{U}:=\text{int}\left(\text{Rot}^{-1}\left(\mathbf{Q}\right)\right)$. We have already argued that $\mathcal{U}$ has full Lebesgue measure. For any $p/q\in\mathbf{Q}\text{ mod }1$, the preimage $\text{rot}^{-1}\left(p/q\right)$ is a closed and non-trivial interval of angles. Since the rotation number only depends on the initial angle, $\text{Rot}^{-1}\left(p/q\right)$ forms a connected set with non-empty interior in $\mathbf{C}^{2}$. Moreover, because the function $\text{Rot}$ is continuous, two different rational values define two distinct connected components of $\mathcal{U}$.

For any rational number $p/q\in\mathbf{Q}\text{ mod }1$, the boundary $\partial\text{rot}^{-1}\left(p/q\right)\cap\left[0,\pi/4\right]$ consists of two points, which implies that $\partial\text{Rot}^{-1}\left(p/q\right)$ is a union of two three-dimensional real manifolds.
    
Similarly, for any $t\in\mathbf{R}\setminus\mathbf{Q}\text{ mod }1$, $\text{rot}^{-1}\left(t\right)\cap\left[0,\pi/4\right]$ consists of a single point, leading to $\text{Rot}^{-1}\left(t\right)$ being a three-dimensional real manifold. For the same reasons as before, the $\text{Rot}^{-1}\left(t\right)$ are exactly the connected components of $\text{Rot}^{-1}\left(\mathbf{R\setminus\mathbf{Q}}\right)$.

\newpage

\bibliographystyle{alpha} % We choose the "plain" reference style
\bibliography{refs}

\end{document}